\documentclass[11pt, a4paper]{article}

\usepackage{amsmath}
\usepackage{amsfonts}
\usepackage{amssymb}
\usepackage[english]{babel}
\usepackage{xcolor}
\usepackage{euscript}
\usepackage{eufrak}
\usepackage{hyperref}
\usepackage{graphicx}

\usepackage{dsfont}

\def\english{\selectlanguage{english}}

\providecommand\mathbb{\bf}
\newcommand\R{{\mathbb R}}

\newcommand\Sp{{\mathbb S}}

\newtheorem{thm}{Theorem}[section]
\newtheorem{lemma}{Lemma}[section]
\newtheorem{pro}{Proposition}[section]
\newtheorem{defi}{Definition}[section]
\newtheorem{coro}{Corollary}[section]
\newtheorem{remark}{Remark}[section]

\newcounter{Remark}

\renewcommand\theRemark{\arabic{Remark}}

\newcounter{steps}
\newenvironment{proof}[1][]{%
\par\medbreak\setcounter{steps}{0}
{\noindent\bfseries Proof#1. }} {\hfill\fbox{\ }\medbreak}

\newcounter{substeps}[steps]

\newcommand{\rot}[0]{
\mathrm{rot}}

\newcommand{\cur}[0]{
\mathrm{curl_x}}

\newcommand{\calE}[0]{
\mathcal{E}}

\newcommand{\calV}[0]{
\mathcal{V}}

\newcommand{\intvt}[1]{
\int _{\R ^3} \!#1 \;\mathrm{d}v}

\newcommand{\Divx}[0]{
\mathrm{div}_x}

\newcommand{\Divv}[0]{
\mathrm{div}_v}

\newcommand{\eps}[0]{
\varepsilon}

\begin{document}
\english

\title{Long time behavior of the Vlasov-Maxwell-Fokker-Planck system with a strong external magnetic field}

\author{
Anh-Tuan VU \thanks{Faculty of Mathematics and Informations, Hanoi University of Science and Technology, 1 Dai Co Viet, Hai Ba Trung, Hanoi, Vietnam. E-mail : {\tt tuan.vuanh1@hust.edu.vn}}
}

\date{(\today)}

\maketitle

\begin{abstract}
The subject matter of this paper concerns the magnetic confinement of plasma. We investigate the long time behavior of the Vlasov-Maxwell-Fokker-Planck system under the effect of a strong external magnetic field. We first formally derive an asymptotically reduced model. In three space dimensions, a constraint occurs along the parallel direction to the magnetic field. To eliminate the corresponding Lagrange multiplier, we then study the Vlasov-Maxwell-Fokker-Planck system in two spatial dimensions and three velocity dimensions $(2d\times 3d)$. In this setting, we show that the obtained asymptotic model defines a well-posed dynamical system, and  we provide a rigorous mathematical justification of the formally derived system. Our proofs rely on the modulated energy method.
\end{abstract}


\\
\\

\section{Introduction}
\label{Intro}
A plasma is a gas that is significantly ionized (through heating or photoionization), consisting of electrons and ions. One of the main applications in plasma physics concerns the energy production through thermo-nuclear fusion. The controlled fusion is achieved by magnetic confinement, where the plasma is confined into toroidal devices called Tokamaks, under the action of strong magnetic fields $\textbf{B}$. This confinement occurs because the radius of the circular motion of charged particles around the magnetic filed lines, known as the Larmor radius $r_L$, is proportional to the inverse of the strength $|\textbf{B}|$ of the magnetic field, i.e., $r_L = mv/|q\textbf{B}|$. Here, $m$ is the particle mass, $q$ is the particle charge and $v$ is the velocity in the plane perpendicular to the magnetic field lines. In Tokamaks, the plasma is heated to extremely high temperatures. As the collision frequency decreases with increasing temperature, fusion plasmas enter a weakly collisional regime. The mathematical modelling useful of such plasma confinement is based on a kinetic framework, which are mesoscopic descriptions for the electrons and ions, coupled with Poisson’s equation or Maxwell’s equation for the computation of the electrostatic or electromagnetic fields, respectively.

In this paper, we study a class of kinetic-type models for fusion plasmas, typically represented by Vlasov-type equations that include self-consistent electromagnetic forces in strongly confined tokamak plasma, as well as in weakly collisional plasma. Direct simulations of the kinetic system of Vlasov equations are fairly expensive from a numerical point of view due to the high dimensionality of the phase space. Additionally, the presence of strong magnetic fields introduces a small time scale, namely the rotation period of particles around the magnetic lines (known as the cyclotronic period $T_c = 2\pi m/|q\textbf{B}|$). This requires  very small time steps for the numerical resolutions. In most industrial applications, it is necessary to reduce the dimensionality of the model, eliminating the unnecessary fast dynamics while retaining the complete low-frequency physics. Therefore, deriving an approximate model that is numerically less expensive is crucial. Large magnetic fields usually lead to the so-called gyro-kinetic or drift-kinetic limits (see \cite{LeeGyro1983, LittHam1981} for physics references and \cite{BosAsyAna, Bre2000} for mathematical results), which consists in the asymptotic behavior of the charged particles dynamics under slowly varying magnetic fields on the typical gyroradius length.

Here, we first derive a new asymptotic model under the assumptions of a strong magnetic field and the large-time asymptotic limit for the three dimensional Vlasov-Maxwell-Fokker-Planck system.  The convergence to this asymptotic model will be rigorously verified for the Vlasov-Maxwell-Fokker-Planck system in two spatial dimensions and three velocity dimensions $(2d\times3d)$.  An analogous problem for the Vlasov-Poisson-Fokker-Planck system has already been carefully studied by the author in recent works. This extension is interesting both from the viewpoint of physics: we are dealing with a more realistic and complete model; and those of mathematics: replacing the Poisson equation by the Maxwell system
we cannot expect too much regularizing effects from the coupling. Let us summarize the previously known results on this topic. In the case of uniform constant magnetic fields, F. Golse and L. Saint-Raymond have  carefully studied this problem for the Vlasov-Poisson system in two dimensions \cite{GolSaintMag1999, Saint2002} (obtained when one restricted to the perpendicular dynamics). In \cite{FilXiSon}, the authors formally derived an asymptotic model from the two dimensional in space and
three dimensional in velocity $(2d\times 3d)$ Vlasov-Maxwell system. For non-constant magnetic fields (i.e., magnetic fields with constant direction but varying amplitude), P. Degond and F. Filbet, in \cite{DegFil16} have formally derived the asymptotic limit of the three dimensional Vlasov-Poisson system. Recently, in \cite{BosTuan}, M. Bostan and A-T. Vu have been provided a rigorous mathematical justification of the obtained asymptotic model for the Vlasov-Poisson-Fokker-Planck system in two dimensions.


\subsection{The physical model}

We consider a one-species plasma, electron for instance, interacting both through the action of the self-consistent electromagnetic field and through collisions, in the presence of an external magnetic field. Let $\tilde{f}=\tilde{f}(\tilde{t},\tilde{x},\tilde{v})$ be the density distribution function of charged particles of mass $m$, charge $q$, depending on time $\tilde{t}\in\R_+$, position $\tilde{x}\in\R^3$, and velocity $\tilde{v}\in\R^3$. The evolution of the plasma is governed by the non-relativistic Vlasov equation  with Fokker-Planck collision operator, and with external magnetic field given by:
\begin{equation}
\label{Non_scale_VMFP}
\dfrac{\partial \tilde{f}}{\partial \tilde{t}} + \tilde{v}\cdot\nabla_{\tilde{x}} \tilde{f} + \dfrac{q}{m}\left( \tilde{E} + \tilde{v}\wedge (\tilde{B} +\tilde{\textbf{B}}_{\text{ext}}) \right)\cdot\nabla_{\tilde{v}} \tilde{f} = \mathcal{Q}(\tilde{f}),\,\, (\tilde{t},\tilde{x},\tilde{v})\in\R_+\times\R^3\times\R^3.
\end{equation}
The self-consistent electromagnetic field $(\tilde{E},\tilde{B})$ satisfies the classical Maxwell system
\begin{equation}
\label{Non_scale_Max1}
\left\{\begin{array}{l}
\mu_0 \epsilon_0 \dfrac{\partial \tilde{E}}{\partial \tilde{t}}  = \mathrm{curl}_{\tilde{x}} \tilde{B} - \mu_0 \tilde{J},\\
\dfrac{\partial \tilde{B}}{\partial \tilde{t}} = - \mathrm{curl}_{\tilde{x}} \tilde{E},
\end{array}\right.
\end{equation}
where $\mu_0$ is the vacuum permeability, and $\epsilon_0$ is the electric permittivity of the vacuum. The electric current is given by
\begin{equation}
\label{Non_scale_curr}
\tilde{J}(\tilde{t},\tilde{x}) = q \tilde{j}(\tilde{t},\tilde{x}) = q \int_{\R^3} \tilde{v} \tilde{f}(\tilde{t},\tilde{x},\tilde{v})\mathrm{d}\tilde{v}.
\end{equation}
Moreover, Maxwell system's is supplemented by Gauss law’s
\begin{equation}
\label{Non_scale_Max2}
\epsilon_0 \mathrm{div}_{\tilde{x}} \tilde{E} = \tilde{\rho},\quad  \mathrm{div}_{\tilde{x}} \tilde{B} =0,
\end{equation}
where $\tilde{\rho}$ is the charge density:
\begin{equation}
\label{Non_scale_charge}
\tilde{\rho}  = q (\tilde{n}- \tilde{D}) = q \left(\int_{\R^3} \tilde{f}(\tilde{t},\tilde{x},\tilde{v})\mathrm{d}\tilde{v} - \tilde{D}(\tilde{x})\right),
\end{equation}
and $\tilde{D}=\tilde{D}(\tilde{x})$ is the density of a background of positive charges, and is assumed to be given. 
The external magnetic field $\tilde{\textbf{B}}_{\text{ext}}$ is of the form
\[
\tilde{\textbf{B}}_{\text{ext}}(\tilde{x}) = \tilde{B}_{\text{ext}}(\tilde{x}) e(\tilde{x}),\quad \mathrm{div}_{\tilde{x}}(\tilde{B}_{\text{ext}}e) =0, \quad |e(\tilde{x}) | =1, \quad \tilde{x}\in\R^3,
\]
for some scalar positive function $\tilde{B}_{\text{ext}}(\tilde{x})$ and some field of unitar vector $e(\tilde{x})$. We assume that $\tilde{B}_{\text{ext}}$, $e$ are smooth.
In equation \eqref{Non_scale_VMFP}, the operator $\mathcal{Q}(\tilde{f})$ is the linear Fokker-Planck operator acting on the velocities, which accounts for friction and diffusion effects, i.e.,
\[
\mathcal{Q}(\tilde{f}) = \dfrac{1}{\tau}\mathrm{div}_{\tilde{v}}{\{\sigma \nabla_{\tilde{v}} \tilde{f} + \tilde{v}\tilde{f}\}},
\]
where $\tau$ is the relaxation time and $\sigma$ is the velocity diffusion, see \cite{Cha1949} for the introduction of this operator, based on the principle of Brownian motion. The system \eqref{Non_scale_VMFP}-\eqref{Non_scale_Max2} is referred as the Vlasov-Maxwell-Fokker-Planck (VMFP) system.  The system is supplemented with the following initial data for the distribution function $\tilde{f}$ and the electro-magnetic field $(\tilde{E},\tilde{B})$
\begin{equation*}
\begin{split}
\tilde{f}(0,\tilde{x},\tilde{v}) = \tilde{f}_{0}(\tilde{x},\tilde{v}),\quad(\tilde{x},\tilde{v})\in\R^3\times\R^3,\\
\tilde{E}(0,\tilde{x}) = \tilde{E}_0 (\tilde{x}),\quad \tilde{B}(0,\tilde{x}) = \tilde{B}_0(\tilde{x}),\quad \tilde{x}\in\R^3.
\end{split}
\end{equation*}
We suppose that initially the plasma is globally neutral, i.e.,
\begin{equation}
\label{equ:global_neutral}
\int_{\R^3}\int_{\R^3} \tilde{f}_0(\tilde{x},\tilde{v})\mathrm{d}\tilde{v}\mathrm{d}\tilde{x} = \int_{\R^3} \tilde{D}(\tilde{x})\mathrm{d}\tilde{x},
\end{equation}
and also that the initial data satisfy the standard compatibility conditions
\begin{equation}
\label{equ:compa_conds}
\mathrm{div}_{\tilde{x}} \tilde{E}_0(\tilde{x}) = q (\int_{\R^3} \tilde{f}_0 (\tilde{x},\tilde{v})\mathrm{d}\tilde{v} -\tilde{D}),\quad \mathrm{div}_{\tilde{x}} \tilde{B}_0(\tilde{x}) =0.
\end{equation}
The VMFP system is widely used in plasma physics, particularly in the study of fusion plasmas, space plasmas, and plasma-based technologies such as plasma propulsion and plasma processings.

At the end of this part, we briefly review some existing works related to the mathematical theory of solutions to the VMFP system. The global existence of renormalized solutions has been obtained by Diperna-Lions in \cite{DipLion}. Regarding the classical solution, we mention the following: In the specific one and one-half dimensional case, the global-in-time existence and uniqueness of classical solutions for the relativistic version of the Vlasov-Maxwell-Fokker-Planck system was proved by \cite{Lai, PanMi}. Both Yu and Yang \cite{YangYu}, Chae \cite{Chae} and Wang \cite{Wang} constructed global classical
solutions to the three-dimensional Vlasov–Maxwell–Fokker–Planck system for initial
data sufficiently close to Maxwellian using Kawashima estimates and the well-known nonlinear energy method, which proposed by Guo \cite{Guo02, Guo03}. However, the global existence of classical solution for the VMFP system in a general framework is a largely open problem. In addition to the global well-posedness, the high electric field limit of the three-dimensional VMFP is derived by Bostan and Goudan in \cite{BosGou08}, while the asymptotic regime of the relativistic VMFP has studied in one and one-half-dimensional framework \cite{BosGou}.

The solution of the VMFP system satisfies the following  fundamental physical properties:
\begin{pro}
Let $\tilde{T}>0$ and let $(\tilde{f},\tilde{E},\tilde{B})$ be a smooth solution of the system VMFP \eqref{Non_scale_VMFP}-\eqref{Non_scale_Max2} with initial data $(\tilde{f}_0,\tilde{E}_0,\tilde{B}_0)$. Then
\begin{enumerate}
\item The solution preserves the particle number on $[0,\tilde{T}]$
\begin{equation}
\label{mass_sm}
\int_{\R^3}\int_{\R^3} \tilde{f}(\tilde{t},\tilde{x},\tilde{v})\mathrm{d}\tilde{v}\mathrm{d}\tilde{x} = \int_{\R^3}\int_{\R^3} \tilde{f}_0(\tilde{x},\tilde{v})\mathrm{d}\tilde{v}\mathrm{d}\tilde{x},
\end{equation}
\item The solution satisfies the free-energy estimate  on $[0,\tilde{T}]$
\begin{equation}
\label{energy_sm}
\begin{split}
& \int_{\R^3}\int_{\R^3}\left(\sigma\ln \tilde{f} + \dfrac{|\tilde{v}|^2}{2} \right)\tilde{f} \mathrm{d}\tilde{v}\mathrm{d}\tilde{x} +  \int_{\R^3}\left(\dfrac{\epsilon_0}{2m}|\tilde{E}|^2 + \dfrac{1}{2\mu_0 m}|\tilde{B}|^2 \right)
\mathrm{d}\tilde{x} \\
&\quad + \dfrac{1}{\tau}\int_0^{\tilde{t}}\int_{\R^3}\int_{\R^3}\dfrac{|\sigma\nabla_{\tilde{v}} \tilde{f} + \tilde{v}\tilde{f}|^2}{\tilde{f}}\mathrm{d}\tilde{v}\mathrm{d}\tilde{x}\mathrm{d}\tilde{s} \\
& \leq   \int_{\R^3}\int_{\R^3}\left(\sigma\ln \tilde{f}_0 + \dfrac{|\tilde{v}|^2}{2} \right)\tilde{f}_0 \mathrm{d}\tilde{v}\mathrm{d}\tilde{x} 
+  \int_{\R^3}\left(\dfrac{\epsilon_0}{2m}|\tilde{E}_0|^2 + \dfrac{1}{2\mu_0 m}|\tilde{B}_0|^2 \right)
\mathrm{d}\tilde{x}.
\end{split}
\end{equation}
\item We have the following local conservation laws on $[0,T]\times\Omega$
\begin{equation}
\label{equ:law_mass}
\partial_{\tilde{t}} \tilde{n} + \mathrm{div}_{\tilde{x}} \tilde{j} =0,
\end{equation}
\begin{equation}
\label{equ:law_moment}
\begin{split}
\partial_{\tilde{t}}\tilde{j}  + \mathrm{div}_{\tilde{x}} \left(\int_{\R^3} \tilde{v}\otimes \tilde{v} \tilde{f} \mathrm{d}\tilde{v} \right)  - \dfrac{q}{m} \left(\tilde{E} \tilde{n} 
+ \tilde{j} \wedge (\tilde{B} + \tilde{\textbf{B}}_{\text{ext}}) \right) 
 = -\dfrac{1}{\tau}\tilde{j}.
\end{split}
\end{equation}
\end{enumerate}
\end{pro}
\subsection{Scaling of the Vlasov-Maxwell-Fokker-Planck system}
We are now interested in deriving the asymptotic limit of the VMFP system \eqref{Non_scale_VMFP}-\eqref{Non_scale_Max2}, which describes the long-time dynamics of charged particles subjected to a strong magnetic field, in order to observe the drift phenomena in the directions orthogonal to the magnetic field. Indeed, it is well known that the velocities of electric cross-field drift, the magnetic gradient drift, and the magnetic curvature drift are proportional to $\frac{1}{B}$. Consequently, it is necessary to observe these drift movements over a large period of time that is proportional to $B$; see \cite{GolSaintMag1999, BosTuan}. Accordingly, we consider the following
\begin{align*}
&\textbf{B}^\eps_{\text{ext}} (x)  = \dfrac{\tilde{\textbf{B}}_{\text{ext}}(\tilde{x})}{\eps},\quad \tilde{f}(\tilde{t},\tilde{x},\tilde{v}) = f^\eps( t,x,v),\quad \tilde{E}(\tilde{t},\tilde{x})=E^\eps(t,x), \quad \tilde{B}(\tilde{t},\tilde{x}) = B^\eps(t,x), \\
&\hspace{45mm} t = \eps\tilde{t} ,\quad x =\tilde x,\quad v =\tilde{v}.
\end{align*}
Here, $\eps >0$ is a small parameter, related to the ratio between the cyclotronic period  and the observation time scale. Hence, 
\[
\partial_{\tilde t} \tilde{f} = \eps \partial_{ t} f^\eps,\quad \partial_{\tilde t} \tilde{E} = \eps \partial_t E^\eps,\quad \partial_{\tilde t} \tilde{B} = \eps \partial_t B^\eps.
\]
Therefore, in equations \eqref{Non_scale_VMFP} and \eqref{Non_scale_Max1}, the
term $\partial_{\tilde t}$ is to be replaced by $\eps \partial_t$; the VMFP system  becomes:
\begin{equation}
\label{equ:VMFP-Scale}
\eps \partial_t f^\eps + v\cdot \nabla_x f^\eps + \dfrac{q}{m}( E^\eps + v\wedge (B^\eps + \dfrac{B_{\text{ext}}e}{\eps}))\cdot\nabla_v f^\eps=\mathcal{Q}(f^\eps) = \dfrac{1}{\tau}\Divv (\sigma \nabla_v f^\eps + v f^\eps),
\end{equation}
\begin{equation}
\label{equ:MaxwellEps}
\mu_0 \epsilon_0 \eps \partial_t E^\eps = \cur B^\eps - \mu_0 J^\eps, \quad \eps \partial_t B^\eps = - \cur E^\eps,
\end{equation}
\begin{equation}
\label{equ:GaussEps}
\epsilon_0 \Divx E^\eps = \rho^\eps= q(n^\eps-D),\quad \Divx B^\eps =0,
\end{equation}
with $\rho^\eps$ and $J^\eps$ given by \eqref{Non_scale_charge} and \eqref{Non_scale_curr}, respectively. 
We complete the above system by applying the following initial conditions:
\begin{equation}
\label{equ:Initial}
f^\eps(0,x,v) = f^\eps_{0}(x,v),\quad
E^\eps(0,x) = E^\eps_0(x),\quad B^\eps(0,x) = B^\eps_0(x),\quad (x,v)\in\R^3\times\R^3,
\end{equation}
which verify \eqref{equ:global_neutral} and \eqref{equ:compa_conds}.
\begin{remark}
\label{re_Maxwell}
The Maxwell equations \eqref{equ:MaxwellEps}-\eqref{equ:GaussEps} can be rewritten in terms of the scalar electrical potential $\Phi^\eps$ and the magnetic vector potential $A^\eps$. The electromagnetic field $(E^\eps,B^\eps)$ is then given by the formulas:
\begin{equation}
\label{elecmag_pot}
E^\eps = -\nabla_x \Phi^\eps - \eps \partial_t A^\eps,\quad B^\eps = \cur A^\eps,
\end{equation}
where the magnetic potential $A^\eps$ verifies the Coulomb's gauge
\[
\Divx A^\eps =0.
\]
Hence, using $\cur B^\eps = \cur\cur A^\eps = -\Delta_x A^\eps$, the electromagnetic potential $(\Phi^\eps,A^\eps)$ satisfies the following equations:
\begin{equation}
\label{equ:pot_em}
\begin{split} \mu_0 \epsilon_0 \eps^2 \partial^2_{tt}A^\eps - \Delta_x A^\eps = \mu_0(J^\eps - \epsilon_0\eps \partial_t\nabla_x\Phi^\eps) ,\\
-\epsilon_0 \Delta_x \Phi^\eps = \rho^\eps,\quad \Divx A^\eps =0.
\end{split}
\end{equation}
\end{remark}
\subsection{Formal derivation of the limit model}
\label{ForDerLimMod}
This section is devoted to deriving the limit model for \eqref{equ:VMFP-Scale}-\eqref{equ:Initial} when $\eps$ becomes
very small, using the properties of the averaged dominant transport operator. At a formal level, we initiate our analysis with a Hilbert expansion 
\begin{equation*}
\begin{split}
f^\eps &= f + \eps f_1 + \eps^2 f_2 + ...\\
E^\eps &= E + \eps E_1 + \eps^2 E_2 + ...\\
B^\eps &= B + \eps B_1 + \eps^2 B_2 + ....
\end{split}
\end{equation*}
 Plugging the above ansatz into the kinetic equation \eqref{equ:VMFP-Scale} and the Maxwell equation \eqref{equ:MaxwellEps}, and then identifying the terms with equal powers in $\eps$, leads to
\begin{align}
\label{equ:Order0}
\eps^{-1}:& \quad\quad(v\wedge B_{\text{ext}}e)\cdot\nabla_v f =0.\\
\label{equ:Order1}
\eps^{0}:& \quad \left\{\begin{array}{l}
 v\cdot \nabla_x f + \dfrac{q}{m}(E + v\wedge B)\cdot \nabla_v f + \dfrac{q}{m}(v\wedge B_{\text{ext}}e)\cdot \nabla_v f_1 = \mathcal{Q}(f),\\
\cur B = \mu_0 q \int_{\R^3}v f \mathrm{d}v,\quad \Divx B =0,\\
\cur E =0,\quad -\epsilon_0 \Divx E = q (\int_{\R^3} f \mathrm{d}v -D).
\end{array}
\right.\\
\label{equ:Order2}
\eps^{1}:&  \quad \left\{\begin{array}{l}
\partial_t f + v\cdot \nabla_x f_1 + \dfrac{q}{m}(E + v\wedge B)\cdot\nabla_v f_1 + \dfrac{q}{m}(E_1+ v\wedge B_1)\cdot\nabla_v f \\+ \dfrac{q}{m}(v\wedge B_{\text{ext}}e)\cdot \nabla_v f_2 = \mathcal{Q}(f_1),\\
\mu_0 \epsilon_0 \partial_t E = \cur B_1 - \mu_0 q \int_{\R^3}v f_1 \mathrm{d}v,\quad \Divx B_1 =0,\\
\partial_t B = -\cur E_1,\quad -\epsilon_0 \Divx E_1 = q \int_{\R^3}f_1\mathrm{d}v.
\end{array}\right.
\end{align}
Multiplying the first equation of \eqref{equ:Order1} by $\sigma(1+ \ln f) + \frac{|v|^2}{2}$ and integrating with respect to $(x,v)\in \R^3\times\R^3$, then using \eqref{equ:Order0} yields
\begin{align}
\label{equ:EquBalancef1}
 \dfrac{1}{\tau}\int_{\R^3}\int_{\R^3}{\dfrac{|\sigma M \nabla_v(f/M)|^2}{f}}\mathrm{d}v\mathrm{d}x = \int_{\R^3}\int_{\R^3}{\dfrac{q}{m}E\cdot vf}\mathrm{d}v\mathrm{d}x,
\end{align}
where $M=M(v) = \frac{e^{-|v|^2/2\sigma}}{(2\pi\sigma)^{3/2}}$ is the Maxwellian distribution in velocity. Integrating the first equation of \eqref{equ:Order1} with respect to $v\in\R^3$, we deduce that $\Divx \intvt{vf} =0$. On the other hand, from $\cur E =0$, it is easy to see that $E = - \nabla_x \Phi$, where the electrical potential $\Phi$ is a solution to the Poisson equation
\[
-\epsilon_0 \Delta_x \Phi  = q(\int_{\R^3}f\mathrm{d}v -D).
\]
This implies that the term in the right hand side of 
\eqref{equ:EquBalancef1} cancels. Therefore we obtain
\[
\dfrac{1}{\tau} \int_{\R^3}\int_{\R^3}{\dfrac{|\sigma M\nabla_v (f/M)|^2}{f}}\mathrm{d}v\mathrm{d}x =0,\,\, t\in \R_+ ,
\]
meaning that $f =nM$, for some function $n=n(t,x)$ to be determined. In that case, the constraint \eqref{equ:Order0} is satisfied and the first equation of \eqref{equ:Order1} becomes
\begin{equation}
\label{equ:constraint}
v\cdot\nabla_x f + \dfrac{q}{m}(E + v\wedge B)\cdot\nabla_v f \in \mathrm{Range}((v\wedge e(x))\cdot\nabla_v),\,\, x\in \Omega.
\end{equation}
For any $b \in \mathbb{S}^2 $, we denote by $\mathcal{R} (\theta,b)$ the rotation of angle $\theta$ around the axis $b$
\begin{equation*}
\mathcal{R}(\theta, b)v = \cos\theta (I_3 -b\otimes b)v - \sin\theta (v\wedge b) + (v\cdot b)b,\,\, v\in\R^3.
\end{equation*}
The characteristic flow of the field $(v\wedge e)\cdot\nabla_v$
\[
\dfrac{\mathrm{d}\calV}{\mathrm{d}\theta} = \calV(\theta;v)\wedge e, \quad \calV(0;v)= v,
\]
is given by
\[
\calV(\theta;v) = \mathcal{R}(-\theta,e)v = \cos\theta (I_3 -e\otimes e)v + \sin\theta (v\wedge e) + (v\cdot e)e,\,\, (\theta,v)\in \R\times\R^3.
\]
For any function $g(v) = (v\wedge e)\cdot\nabla_v h$ in the range of the operator $(v\wedge e)\cdot\nabla_v$, we have
\[
g(\calV(\theta;v)) = \dfrac{\mathrm{d}}{\mathrm{d}\theta}h(\calV(\theta;v)),\,\, (\theta,v)\in \R\times\R^3,
\]
and by the periodicity of the flow we obtain
\[
\dfrac{1}{2\pi}\int_{0}^{2\pi} g(\calV(\theta;v)) \mathrm{d}\theta =0, \,\, v\in\R^3.
\]
Therefore, for any $x\in\R^3$, the average along the characteristic flow with respect to $(v\wedge e(x))\cdot\nabla_v$ of the function $v\cdot\nabla_x f + \frac{q}{m}(E+v\wedge B)\cdot\nabla_v f$ vanishes. But
\begin{equation*}
\begin{split}
(v\wedge B)\cdot\nabla_v f &=0,\\
v\cdot\nabla_x f + \frac{q}{m}E\cdot\nabla_v f & = \dfrac{n}{\sigma}Mv\cdot\left[\sigma\dfrac{\nabla_x n}{n} - \dfrac{q}{m}E\right],
\end{split}
\end{equation*}
and since
\[
\dfrac{1}{2\pi} \int_{0}^{2\pi} M(\calV(\theta;v))\calV(\theta;v)\mathrm{d}\theta = M(v)(v\cdot e)e,
\]
finally, from \eqref{equ:constraint}, we obtain the following constraint
\[
e\cdot \sigma\dfrac{\nabla_x n}{n} =  \dfrac{q}{m}e\cdot E,\,\, x\in\R^3.
\]
Multiplying the first equation of \eqref{equ:Order1} by $v$ and integrating with respect to $v\in\R^3$ we obtain
\[
\Divx \int_{\R^3}{v\otimes v f}\mathrm{d}v - n\dfrac{q}{m}E - \omega_c \int_{\R^3}{v f_1\wedge e}\mathrm{d}v =0,\quad \omega_c(x) = \dfrac{qB_{\text{ext}}(x)}{m}.
\]
Since $f$ is a Maxwellian equilibrium, we have $\intvt{v\otimes v f} = \sigma n I_3$ and the previous equality becomes
\[
\omega_c \int_{\R^3}{v f_1}\mathrm{d}v \wedge e = \sigma \nabla_x n - n\dfrac{q}{m}E,
\]
or equivalently
\begin{align*}
 (I_3 - e\otimes e)\int_{\R^3}{v f_1}\mathrm{d}v = \dfrac{ne}{\omega_c}\wedge k[n],\quad  k[n]= \sigma\dfrac{\nabla_x n}{n} - \dfrac{q}{m}E.
\end{align*}
Combing the above arguments, therefore from \eqref{equ:Order1}, we obtain that
\begin{equation}
\label{equ:Order1_bis}
\left\{\begin{array}{l}
e\cdot k[n] =0 ,\quad k[n] = \sigma\dfrac{\nabla_x n}{n} - \dfrac{q}{m} E,\\
(I_3 - e\otimes e)\int_{\R^3}vf_1\mathrm{d}v = \dfrac{ne}{\omega_c}\wedge k[n],\\
\cur B =0, \quad \Divx B =0,\\
\cur E =0, \quad -\epsilon_0 \Divx E = q (n -D(x)).
\end{array}\right.
\end{equation}
The condition $\Divx B = 0$ implies that there exists a magnetic vector potential $A$ such that $B = \cur A$. Then, from $\cur B = 0$, we have $\cur\cur A = 0$ leading to $\int_{\R^3}|\cur A|^2\mathrm{d}x =0$. Hence, we obtain $B =\cur A =0$.
\begin{remark} Performing Hilbert expansions
\begin{align*}
\Phi^\eps = \Phi + \eps \Phi_1 +\eps^2 \Phi_2 + ...\\
A^\eps = A + \eps A_1 + \eps^2 A_2 + ...,
\end{align*}
we also obtain from equations \eqref{elecmag_pot} and \eqref{equ:pot_em} that  $E = -\nabla_x\Phi$ and $B = \cur A$ with
\begin{align*}
-\epsilon_0 \Delta_x \Phi = \rho = q(n -D),\\
-\Delta_x A = \mu_0 J =0,\quad \Divx A =0.
\end{align*} 
\end{remark}

The time evolution for the concentration $n$ comes by integrating the first equation of \eqref{equ:Order2} with respect to $v\in\R^3$
\begin{equation}
\label{equ:EquContinum}
\partial_t n + \Divx\int_{\R^3}{v f_1}\mathrm{d}v =0.
\end{equation}
Using the second equation of \eqref{equ:Order1_bis}, the divergence with respect to $x$ of $\int_{\R^3}{vf_1}\mathrm{d}v$ writes
\begin{align*}
\Divx \int_{\R^3}{v f_1}\mathrm{d}v &= \Divx \left[ (I_3 - e\otimes e)\int_{\R^3}{v f_1}\mathrm{d}v \right] + \Divx \left[ e\otimes e\intvt{v f_1} \right]\\
&= \Divx\left( \dfrac{n e}{\omega_c}\wedge k[n]\right) +  B_{\text{ext}}e\cdot\nabla_x \int_{\R^3}{\dfrac{(v\cdot e)f_1}{B_{\text{ext}}}}\mathrm{d}v.
\end{align*}
Coming back in \eqref{equ:EquContinum}, we obtain that the limiting concentration $n$ satisfies the following equation 
\[
\partial_t n + \Divx\left( \dfrac{n e}{\omega_c}\wedge k[n]\right) +  B_{\text{ext}}e\cdot\nabla_x p =0,\quad p=\int_{\R^3}{\dfrac{(v\cdot e)f_1}{B_{\text{ext}}}}\mathrm{d}v.
\]
Therefore, from \eqref{equ:Order2}, we have
\begin{equation}
\label{equ:Order2_bis}
\left\{\begin{array}{l}
\partial_t n + \Divx\left( \dfrac{n e}{\omega_c}\wedge k[n]\right) +  B_{\text{ext}}e\cdot\nabla_x p =0,\\
\mu_0\epsilon_0 \partial_t E = \cur B_1 - \mu_0 q \dfrac{n e}{\omega_c}\wedge k[n] - \mu_0 q B_{\text{ext}}e p,\\
  \Divx B_1=0.
 \end{array}\right.
\end{equation}
Combining \eqref{equ:Order1_bis} and \eqref{equ:Order2_bis}, we obtain the limit model of \eqref{equ:VMFP-Scale}-\eqref{equ:GaussEps} as $\eps\searrow 0$ given by
\begin{equation}
\label{equ:lim_mod_3d}
\left\{\begin{array}{l}
\partial_t n + \Divx\left( \dfrac{n e}{\omega_c}\wedge k[n] \right) +  B_{\text{ext}}e\cdot\nabla_x p =0,\\
e\cdot k[n] =0,\quad k[n] = \sigma \dfrac{\nabla_x n}{n} -\dfrac{q}{m}E,\\
\cur E = 0,\quad \epsilon_0 \Divx E = q (n -D),\\
\mu_0\epsilon_0 \partial_t E = \cur B_1 - \mu_0 q \dfrac{n e}{\omega_c}\wedge k[n]
 - \mu_0 q B_{\text{ext}}e p,\\
 \Divx B_1=0.
 \end{array}\right.
\end{equation}
The limit model involves a Lagrange multiplier $p$, associated to the constraint. 
\begin{remark} The asymptotic model \eqref{equ:lim_mod_3d} describes the non-trivial electron dynamics perpendicular to the magnetic field. The well-posedness of this problem is an interesting issue and is left for further works.
\end{remark}

We now provide  some fundamental properties satisfied by the asymptotic model \eqref{equ:lim_mod_3d}. Formally, we have the following balances:
\begin{pro}
\label{BalLiMod}
Any non-negative regular solution of the limit model \eqref{equ:lim_mod_3d} verifies the mass and free energy conservations
\[
\dfrac{\mathrm{d}}{\mathrm{d}t}\int_{\R^3}{n(t,x)}\mathrm{d}x =0,\quad \dfrac{\mathrm{d}}{\mathrm{d}t}\int_{\R^3}{\left\{\sigma n\ln n + \dfrac{\epsilon_0}{2m}|E|^2\right\}}\mathrm{d}x =0.
\]
\end{pro}
\begin{proof}
Clearly we have the total mass conservation. For the energy conservation, first, multiplying the first equation of \eqref{equ:lim_mod_3d} by $\sigma(1+ \ln n)$ and integrating with respect to $x$, we have
\begin{equation}
\label{equ:entro_lim_mod}
\dfrac{\mathrm{d}}{\mathrm{d}t} \int_{\R^3} \sigma n\ln n \mathrm{d}x = \int_{\R^3} \left[ \dfrac{ne}{\omega_c}\wedge \left( \sigma \dfrac{\nabla_x n}{n} - \dfrac{q}{m} E\right)\right]\cdot \sigma \dfrac{\nabla_x n}{n} \mathrm{d}x + \int_{\R^3} B_{\text{ext}}e p \cdot \sigma\dfrac{\nabla_x n}{n}\mathrm{d}x.
\end{equation}
Then, multiplying the fourth equation of \eqref{equ:lim_mod_3d} with $\frac{E}{\mu_0 m}$ and integrating with respect to $x$, we obtain
\begin{equation}
\label{equ:elec_ener_mod}
\begin{split}
\dfrac{\epsilon_0}{2m} \dfrac{\mathrm{d}}{\mathrm{d}t}\int_{\R^3}|E|^2\mathrm{d}x &= \dfrac{1}{\mu_0 m}\int_{\R^3}\cur B_1\cdot E \mathrm{d}x -  \dfrac{q}{m} \int_{\R^3} \left[ \dfrac{ne}{\omega_c}\wedge \left( \sigma \dfrac{\nabla_x n}{n} - \dfrac{q}{m} E\right)\right]\cdot  E \mathrm{d}x\\
& \quad -  \dfrac{q}{m}\int_{\R^3} B_{\text{ext}}e p\cdot E \mathrm{d}x.
\end{split}
\end{equation}
Using $\cur E =0$ yields $\int_{\R^3} \cur B_1\cdot E \mathrm{d}x =0$ and then using the second equation of \eqref{equ:lim_mod_3d}, we deduce from \eqref{equ:entro_lim_mod} and \eqref{equ:elec_ener_mod} the energy conservation of the limit system.
\end{proof}

Recall the usual drift velocities when dealing with magnetic confinement: the electric field drift, the magnetic gradient drift, and the magnetic curvature drift
\[
\dfrac{E\wedge e}{B}, \,\, - \dfrac{m|v\wedge e|^2}{2qB}\dfrac{\nabla_x B\wedge e}{B} = - \dfrac{|v\wedge e|^2}{2}\dfrac{\nabla_x \omega_c \wedge e}{\omega_c ^2},\,\, -\dfrac{m|v\wedge e|^2}{qB}\partial_x e e\wedge e = -\dfrac{(v\cdot e)^2}{\omega_c}\partial_x e e\wedge e.
\]
When working at the fluid level, the averages with respect to $v\in\R^3$ of the above drift velocities become
\[
v_{\wedge D} = \int_{\R^3}{\dfrac{E\wedge e}{B}M(v)}\mathrm{d}v = \dfrac{E\wedge e}{B},
\]
\[
v_{GD} = - \int_{\R^3}{\dfrac{|v\wedge e|^2}{2}\dfrac{\nabla_x \omega_c \wedge e}{\omega_c ^2}M(v)}\mathrm{d}v = -\sigma \dfrac{\nabla_x \omega_c \wedge e}{\omega_c ^2},
\]
\[
v_{CD} = - \int_{\R^3}{\dfrac{(v\cdot e)^2}{\omega_c}\partial_x e e\wedge e M(v)}\mathrm{d}v = - \sigma \dfrac{\partial_x e e\wedge e}{\omega_c}.
\]
The flux in the limit model \eqref{equ:lim_mod_3d} also writes $n\calV[n]$, where $\calV[n] = v_{\wedge D} + v_{GD} + v_{CD}$.
\begin{pro}
\label{Equiv_form}
Any non-negative regular function $n$ satisfying
\[
\partial_t n + \Divx\left[ \dfrac{ne}{\omega_c}\wedge \left(\sigma\dfrac{\nabla_x n}{n}-\dfrac{q}{m}E\right) \right] + Be\cdot\nabla_x p =0,
\]
also verifies
\[
\partial_t n + \Divx(n\calV[n]) + Be\cdot\nabla_x\tilde{p} =0,\quad \calV[n] = \dfrac{E\wedge e}{B} -\sigma \dfrac{\nabla_x \omega_c \wedge e}{\omega_c ^2}- \sigma \dfrac{\partial_x e e\wedge e}{\omega_c},
\]
and $\tilde{p} = p+ \frac{\sigma n}{B\omega_c}(e\cdot \rot_x e)$.
\end{pro}
\begin{proof}
Recall the formula $\Divx{(\xi \wedge \eta)}= \eta\cdot \rot_x \xi - \xi\cdot \rot_x \eta$, for any smooth vector fields $\xi$ and $\eta$. Therefore we can write
\begin{align*}
 & \Divx\left[\dfrac{ne}{\omega_c}\wedge\left( \sigma\dfrac{\nabla_x n}{n} -\dfrac{q}{m}E \right) \right]\\
&=\Divx\left( n\dfrac{E\wedge e}{B}\right) + \sigma\, \rot_x\left(\dfrac{e}{\omega_c} \right)\cdot\nabla_x n\\
&= \Divx\left( n\dfrac{E\wedge e}{B}\right) + \sigma\, \Divx\left(n \,\rot_x \left(\dfrac{e}{\omega_c} \right) \right)\\
&= \Divx\left( n\dfrac{E\wedge e}{B}\right) - \sigma\, \Divx\left( n\dfrac{\nabla_x \omega_c \wedge e}{\omega_c ^2} - \dfrac{n}{\omega_c}\rot_x e \right)\\
&= \Divx\left( n v_{\wedge D} + n v_{GD} + \dfrac{\sigma n}{\omega_c}(I_3 -e\otimes e)\rot_x e\right) + \Divx\left(\dfrac{\sigma n}{\omega_c}(e\cdot\rot_x e)e \right).
\end{align*}
Notice that we can write 
\[
(I_3 -e\otimes e)\rot_x e = e\wedge (\rot_x e \wedge e) = e\wedge[(\partial_x e - {^t}\partial_x e)e] = e \wedge \partial_x e e,
\]
implying that
\[
\sigma \dfrac{n}{\omega_c} (I_3 - e\otimes e)\rot_x e = - \sigma \dfrac{n}{\omega_c} \partial_x e e \wedge e = n v_{CD}.
\]
Finally, we obtain
\[
\Divx\left( \dfrac{ne}{\omega_c}\wedge \nabla_x k[n] \right) = \Divx(n\calV[n]) + Be\cdot \nabla_x\left[ \dfrac{\sigma n}{B\omega_c}(e\cdot \rot_x e) \right],
\]
and our conclusion follows.
\end{proof}
\subsection{Notation} Let us introduce several notations used throughout the present work.
For any nonnegative integer $k$ and $p \in [1, \infty]$, $W^{ k,p} = W^{ k,p} (\Omega)$ stands for the $k$-th order $L^p$ Sobolev space. $C_b^k$ stands for $k$ times continuously differentiable functions, whose partial derivatives, up to order $k$, are all bounded and $C^k ([0, T ]; E)$ is the set of $k$-times continuously differentiable functions from an interval $[0, T] \subset \R$ into a Banach space $E$. $L^p(0, T ; E)$ is the set of measurable functions from an interval $(0, T )$ to a Banach space $E$, whose $p$-th power of the $E$-norm is Lebesgue measurable. Finally we denote by $\mathcal{M}(\Omega)$ the set of Randon measures on $\Omega$. 

\section{Main results}
\label{main_result}
\subsection{Renormalized solutions for the Vlasov-Maxwell-Fokker-Planck system with external magnetic field}
We will investigate the long time behavior of the VMFP equations \eqref{equ:VMFP-Scale}-\eqref{equ:Initial} when the external magnetic field becomes strong. Motivated by this, we are looking for global-in-time solutions satisfying local conservation laws and uniform bounds with respect to the magnetic field. The global existence of classical solution of the VMFP system is still an open question in a general framework. Global-in-time existence weak solutions is well constructed but it is not completely satisfactory in this work since it does not ensure the free-energy estimate \eqref{energy_sm} and the local conservation of momentum \eqref{equ:law_moment} in the sense of distributions on $[0,T[\times\R^3$. This is one of the main difficulties in rigorously deriving hydrodynamic models from the VMFP system. Therefore, following \cite{PueSaint2004, BosGou08}, we will establish the results in the framework of renormalized solutions for the VMFP equations as introduced by Diperna and Lions \cite{DipLion}. Their construction yields a solution which satisfies in addition a conservation
law of momentum and a free-energy inequality with defect measures.

We introduce notions of renormalized solutions to the VMFP equations \eqref{equ:VMFP-Scale}-\eqref{equ:Initial} for any fixed $\eps>0$.
\begin{defi}
\label{weak_sol_VMFP}
Let $T>0$. A renormalized solution on the time interval $[0,T]$ is a triple weak solutions
\[
(f^\eps, E^\eps, B^\eps)\in  L^\infty(0,T;L^1\cap L^2(\R^3\times\R^3))\times L^\infty(0,T;L^2(\R^3))^6,
\]
which satisfies equations \eqref{equ:VMFP-Scale}-\eqref{equ:Initial} in the sense of distributions. Moreover, it verifies the local conservation law of mass \eqref{equ:law_mass} (also in the sense of distributions), the local conservation law of momentum in \eqref{equ:law_moment}, and the free-energy estimate \eqref{energy_sm} with defect measures, as follows:
\begin{itemize}
\item The local conservation law of mass in the sense of distributions on $[0,T[\times\R^3$:
\begin{equation}
\label{equ:law_charge_meas}
\partial_t \int_{\R^3} f^\eps \mathrm{d}v + \dfrac{1}{\eps}\Divx\int_{\R^3} v f^\eps \mathrm{d}v =0. 
\end{equation}
\item The local conservation law of the momentum with symmetric nonnegative matrix-valued defect measures $\mu^\eps_E$, $\mu^\eps_B\in L^\infty(0,T;\mathcal{M}(\R^3))^9$, a vector-valued defect measure $\mu^\eps_{EB}\in L^\infty(0,T;\mathcal{M}(\R^3))^3$, and a defect measure $m^\eps=m^\eps(t,x,\sigma)\in L^\infty(0,T;\mathcal{M}(\R^3\times\Sp^2))$
\begin{equation}
\label{equ:law_moment_meas}
\begin{split}
&\eps \partial_t \int_{\R^3} vf^\eps\mathrm{d}v + \Divx \left( \int_{\R^3}v\otimes v f^\eps\mathrm{d}v + \int_{\Sp^2}\sigma\otimes\sigma \mathrm{d}m^\eps\right) -\dfrac{\epsilon_0}{m}\mathcal{A}(E^\eps,E^\eps) - \dfrac{q}{m} D E^\eps  
\\&  \quad - \dfrac{1}{m\mu_0}\mathcal{A}(B^\eps,B^\eps) 
 - \dfrac{\epsilon_0}{m}\Divx\left( \mu^\eps_E - \dfrac{1}{2}tr(\mu^\eps_E)I_3 \right) - \dfrac{1}{m\mu_0}\left(\mu^\eps_B - \dfrac{1}{2}tr(\mu^\eps_B)I_3 \right) \\
&\quad + \dfrac{\epsilon_0\eps}{m}\partial_t (E^\eps\wedge B^\eps) + \dfrac{\epsilon_0\eps}{m}\partial_t \mu^\eps_{EB} - \dfrac{q}{m}\dfrac{1}{\eps}\int_{\R^3} v f^\eps\mathrm{d}v \wedge (B_{\text{ext}}e) \\
&=- \dfrac{1}{\tau}\int_{\R^3} v f^\eps \mathrm{d}v,
\end{split}
\end{equation}
in the sense of distributions on $[0,T[\times\R^3$, where
\begin{equation}
\label{equ:quad_field}
\mathcal{A}(u,u) = \Divx\left(u\otimes u - \dfrac{|u|^2}{2}I_3\right) = u\Divx u - u\wedge \cur u.
\end{equation}
\item The free-energy estimate with the defect measures $\mu^\eps_E$, $\mu^\eps_B$ and $m^\eps$ for all $0<t\leq T$
\begin{equation}
\label{equ:ener_meas}
\begin{split}
& \int_{\R^3}\int_{\R^3}\left(\sigma\ln f^\eps + \dfrac{|v|^2}{2} \right)f^\eps \mathrm{d}v\mathrm{d}x +  \int_{\R^3}\left(\dfrac{\epsilon_0}{2m}|E^\eps|^2 + \dfrac{1}{2\mu_0 m}|B^\eps|^2 \right)\mathrm{d}x\\
&\quad + \dfrac{1}{\tau\eps}\int_0^t\int_{\R^3}\int_{\R^3}\dfrac{|\sigma\nabla_v f^\eps+ vf^\eps|^2}{f^\eps}\mathrm{d}v\mathrm{d}x\mathrm{d}s + \dfrac{1}{2}\int_{\R^3}\int_{\Sp^2}\mathrm{d}m^\eps \\
&\quad+ \dfrac{1}{2}\int_{\R^3} \mathrm{d}\left(\dfrac{\epsilon_0}{m}tr(\mu^\eps_E) + \dfrac{1}{\mu_0 m}tr(\mu^\eps_B)\right) \\
&\leq \int_{\R^3}\int_{\R^3}\left(\sigma\ln f^\eps_0 + \dfrac{|v|^2}{2} \right)f^\eps_0 \mathrm{d}v\mathrm{d}x 
+  \int_{\R^3}\left(\dfrac{\epsilon_0}{2m}|E^\eps_0|^2 + \dfrac{1}{2\mu_0 m}|B^\eps_0|^2 \right)
\mathrm{d}x .
\end{split}
\end{equation}
\end{itemize}
\end{defi}

\subsection{Asymptotic of the $2d\times 3d$ VMFP system}
To rigorously drive the asymptotic model from \eqref{equ:VMFP-Scale}-\eqref{equ:Initial}, let us set our assumptions:
\begin{description}
\item[\textbf{ASSUMPTION 1}]: The external magnetic field only applies in the $x_3$-direction:
\[
B_{\text{ext}} e = (0,0,B_{\text{ext}}(x)),\quad e(x) = (0,0,1).
\]
From the divergence free condition $\Divx( B_{\text{ext}}e) =0$, we deduce that $B_{\text{ext}}(x)$ only depends on $x_1,x_2$.
\item[\textbf{ASSUMPTION 2}]: The plasma is homogeneous in the direction parallel to the applied magnetic field $x_3$. Hence, the distribution function $f^\eps$ and the electromagnetic field $(E^\eps,B^\eps)$ do not depend on $x_3$, as indicated below:
\[
f^\eps:\R^+_t\times\R^2_{x_1,x_2}\times\R^3_v\to\R^+,\quad E^\eps:\R^+_t\times\R^2_{x_1,x_2}\to \R^3,\,\,\text{and}\,\, B^\eps:\R^+_t\times\R^2_{x_1,x_2}\to \R^3.
\]
\end{description}

Under these assumptions, the Vlasov-Fokker-Planck equation \eqref{equ:VMFP-Scale} can be rewritten in two spatial dimensions as follows:
\begin{equation}
\label{equ:VMFP-Scale2dx}
\begin{split}
\eps \partial_t f^\eps + v\cdot\nabla_x f^\eps + \dfrac{q}{m}\left(E^\eps + v\wedge B^\eps + \dfrac{B_{\text{ext}}}{\eps}{^\perp v}\right)\cdot\nabla_v f^\eps = \dfrac{1}{\tau}\Divv(\sigma\nabla_v f^\eps + vf^\eps),
\end{split}
\end{equation}
where $f^\eps = f^\eps(t,x_\perp,v)$ with $x_\perp = (x_1,x_2)$ for any $x=(x_1,x_2,x_3)^t\in\R^3$, and  $^\perp v = (v_2,-v_1,0)$ for any $v=(v_1,v_2,v_3)^t\in\R^3$.

The limit model of \eqref{equ:VMFP-Scale2dx}, \eqref{equ:MaxwellEps}-\eqref{equ:Initial} as $\eps\searrow 0$ becomes as follows:
\begin{equation}
\label{equ:lim_mod_2dx}
\left\{\begin{array}{l}
\partial_t n(t,x_\perp) + \Divx\left( \dfrac{n e}{\omega_c}\wedge k[n] \right)  =0,\quad k[n] = \sigma\dfrac{\nabla_x n}{n} - \dfrac{q}{m}E,\\
e\cdot E =0, \quad E= E(t,x_\perp) = (E_1, E_2,E_3)(t,x_\perp),\\
\epsilon_0 \Divx E = q(n -D),\quad \cur E =0,\\
\mu_0 \epsilon_0 \partial_t E = \cur B_1 - \mu_0 q \dfrac{n e}{\omega_c}\wedge k[n],\quad \Divx B_1 =0,\quad B_1 =B_1(t,x_\perp),
\end{array}\right.
\end{equation}
with initial conditions
\begin{equation}
\label{equ:init_lim_mod_2dx}
\begin{split}
n(0,x_\perp) = n_0(x_\perp) =\int_{\R^3}f(0,x_\perp,v)\mathrm{d}v,\quad E(0,x_\perp) = E^0(x_\perp),\quad B_1(0,x_\perp)=B^0_1(x_\perp),\\
\epsilon_0 \Divx  E^0 = q(n_0 -D),\quad \cur E^0 =0.
\end{split}
\end{equation}
\subsection{Strong convergence for well-prepared initial data}
In the present work, the convergence of the VMFP system \eqref{equ:VMFP-Scale2dx}, \eqref{equ:MaxwellEps}-\eqref{equ:Initial} towards the limit model \eqref{equ:lim_mod_2dx} will be investigated by appealing to the relative entropy or modulated energy method, as introduced in \cite{Yau1991}. To the best of our knowledge, to date, there has been no result on the asymptotic
regime for the Vlasov-Maxwell-Fokker-Planck system with non-uniform external magnetic field. This technique yields strong convergence, provided that the solution of the limit system is regular and that  the initial data converge appropriately. Many asymptotic regimes were obtained using this technique, see \cite{Bre2000, BreMauPue2003, GolSaintQuas2003, PueSaint2004} for quasineutral regimes in collisionless plasma physics, \cite{Saint2003, BerVas2005} for hydrodynamic limits in gaz dynamics, \cite{GouJabVas2004} for fluid-particle interaction, \cite{BosGou08, Bos2007} for high electric or magnetic field limits in plasma physics. 

Before writing our main result, we define the modulated energy $\calE[n^\eps(t)|n(t)]$ by
\begin{equation}
\label{equ:mod_energy}
\begin{split}
\calE[n^\eps(t)|n(t)] = \sigma \int_{\R^2}{n(t) h\left( \dfrac{n^\eps(t)}{n(t)}\right)}\mathrm{d}x_\perp + \dfrac{\epsilon_0}{2m}\int_{\R^2}{|E^\eps -E|^2}\mathrm{d}x_\perp + \dfrac{1}{2\mu_0 m}\int_{\R^2}|B^\eps - \eps B_1|^2\mathrm{d}x_\perp,
\end{split}
\end{equation}
where $h: \R_+ \to \R_+$ is the convex function defined by $h(s) = s\ln s -s +1$, $s\in \R_+$. This quantity splits into the standard $L^2$ norm of the electromagnetic field plus the relative entropy between the particle density $n^\eps$ of \eqref{equ:VMFP-Scale2dx}, \eqref{equ:MaxwellEps}-\eqref{equ:Initial} and the particle concentration $n$ of the limit model \eqref{equ:lim_mod_2dx}. The main result of this paper is the following
\begin{thm}
\label{MainThm}
Let $B_{\text{ext}}$ be a smooth magnetic field, such that $\inf_{x_\perp\in\R^2}B_{\text{ext}}(x_\perp)=B_0 >0$. 
Assume that the initial particle densities $(f^\eps_{0})_{\eps>0}$ and electromagnetic fields $(E^\eps_0,B^\eps_0)_{\eps>0}$ satisfy $f^\eps_{0}\geq 0$, $M_{0}:=\sup_{\eps>0}M^\eps_{0}<+\infty$, $U_{0}:=\sup_{\eps>0}U^\eps_{0}<+\infty$ where
\[
M^\eps _{0} := \int_{\R^2}\int_{\R^3}{f^\eps _{0} (x_\perp,v)}\mathrm{d}v\mathrm{d}x_\perp ,\]
\[
 U^\eps _{0} := \int_{\R^2}\int_{\R^3}{\dfrac{|v|^2}{2}f^\eps _{0} (x_\perp,v)}\mathrm{d}v\mathrm{d}x_\perp + \int_{\R^2}\left(  \dfrac{\epsilon_0}{2m}|E^\eps_0|^2 + \dfrac{1}{2\mu_0 m}|B^\eps_0|^2\right)\mathrm{d}x_\perp.
\]
Let $T>0$. We denote by $(f^\eps,E^\eps, B^\eps)_{\eps>0}$ the solutions of equations \eqref{equ:VMFP-Scale}-\eqref{equ:Initial} in the sense of Definition \ref{weak_sol_VMFP} with initial data $(f^\eps_0,E^\eps_0,B^\eps_0)$ on $[0,T]$, as provided by Theorem \ref{exis_weak_sol}. Let $(n,E,B_1)$ be a unique regular solution of \eqref{equ:lim_mod_2dx} on $[0,T]$ with the initial conditions $(n_0,E^0,B^0_1)$, as constructed in Propositions \ref{existsol_lim_mod1} and \ref{existsol_lim_mod2}. Then, we have the following inequality for $0<\eps<1$ and $t\leq T$:
\begin{equation*}
\begin{split}
\calE[n^\eps(t)|n(t)] +&  \sigma\int_{\R^2}\int_{\R^3}{n^\eps(t) M h\left(\dfrac{f^\eps(t)}{n^\eps(t) M} \right)}\mathrm{d}v\mathrm{d}x_\perp + \dfrac{1}{2\eps\tau}\int_0^t\int_{\R^2}\int_{\R^3}{\dfrac{|\sigma \nabla_v f^\eps + v f^\eps|^2}{f^\eps}}\mathrm{d}v\mathrm{d}x_\perp\mathrm{d}s\\
& +  \dfrac{1}{2}\int_{\R^2}\int_{\Sp^2}\mathrm{d}m^\eps+
\dfrac{1}{2}\int_{\Omega} \mathrm{d}\left( \dfrac{\epsilon_0}{m}tr(\mu^\eps_E) + \dfrac{\mu_0}{m}tr(\mu^\eps_B)\right)\\
&\leq  \left[\calE[n^\eps_0|n_0] +  \sigma\int_{\R^2}\int_{\R^3}{n^\eps_0 M(v) h\left(\dfrac{f^\eps_0}{n^\eps_0 M(v)} \right)}\mathrm{d}v\mathrm{d}x_\perp + C \sqrt{\eps}\right]e^{C t},
\end{split}
\end{equation*}
where $n^\eps_{0} = \int_{\R^3}{f^\eps_0}\mathrm{d}v$, and $C>0$ is a constant that is independent of $\eps$.
In particular, if
\[
\lim_{\eps\searrow 0}\sigma \int_{\R^2}\int_{\R^3}{n^\eps_0}M(v) h\left(\dfrac{f^\eps_0}{n^\eps_0 M(v)} \right)\mathrm{d}v\mathrm{d}x_\perp =0,\quad \lim_{\eps\searrow 0}\calE[n^\eps_{0}|n_{0}] =0 ,
\]
then we obtain 
\[
\lim_{\eps\searrow 0} \sup_{0 \leq t \leq T} \sigma \int_{\R^2}\int_{\R^3}{n^\eps(t)M(v)h\left(\dfrac{f^\eps}{n^\eps M} \right)}\mathrm{d}v\mathrm{d}x_\perp = 0,\quad \lim_{\eps\searrow 0} \sup_{0 \leq t \leq T}\calE[n^\eps(t)|n(t)] =0,
\]
\[
\lim_{\eps\searrow 0}\dfrac{1}{\eps \tau}\int_0^T\int_{\R^2}\int_{\R^3}{\dfrac{|\sigma \nabla_v f^\eps + v f^\eps |^2}{f^\eps}} \mathrm{d}v\mathrm{d}x_\perp\mathrm{d}t =0.
\]
Moreover, we have the convergences $\lim_{\eps\searrow 0} f^\eps = nM$ in $L^\infty (0,T;L^1(\R^2\times \R^3))$, and $\lim_{\eps\searrow 0} E^\eps = E$ and $\lim_{\eps\searrow 0} B^\eps -\eps B_1 =0$ in $L^\infty(0,T;L^2(\R^2))$.
\end{thm}

\begin{remark}
The assumption on the initial data is necessary with the method we use,
which also requires the smoothness of the solution of the limit equations. In some sense this assumption means that
the initial state is already close to the shape of the limit.
\end{remark}

The remainder of the paper is organized as follows. In Section \ref{Prelim}, we establish some a priori estimates for the three-dimensional VPFP system. In Section \ref{lim_mod_2d}, we investigate the well-posedness of the limiting model in two spatial dimensions. Finally, in Section \ref{Conve}, we rigorously prove the convergence towards the asymptotic model.

\section{Preliminaries}
\label{Prelim}
We now state the result on the global existence solution to the $3d\times 3d$ VMFP system:
\begin{thm}
\label{exis_weak_sol}
Let $T>0$ and $B_{\text{ext}}\in L^\infty(\R^3)$. For fixed $\eps>0$, let $f^\eps_0\in L^1\cap L^2(\R^3\times\R^3)$ be an nonnegative function, and $(E^\eps_0,B^\eps_0)\in L^2(\R^3)^2$ be two vector fields satisfying the energy bound
\begin{equation}
\label{equ:init_ener_bound}
 \int_{\R^3}\int_{\R^3}\left(\sigma |\ln f^\eps_0 | + \dfrac{|v|^2}{2} \right)f^\eps_0 \mathrm{d}v\mathrm{d}x 
+  \int_{\R^3}\left(\dfrac{\epsilon_0}{2m}|E^\eps_0|^2 + \dfrac{1}{2\mu_0 m}|B^\eps_0|^2 \right)
\mathrm{d}x <+\infty.
\end{equation}
Then there exists a weak solution $(f^\eps, E^\eps, B^\eps)$ to \eqref{equ:VMFP-Scale}-\eqref{equ:Initial} in the sense of Definition \ref{weak_sol_VMFP} which satisfies the local conservation laws \eqref{equ:law_charge_meas}-\eqref{equ:law_moment_meas}, and the free-energy estimate \eqref{equ:ener_meas}.
\end{thm}

The proof of Theorem \ref{exis_weak_sol}, concerning the global existence of solutions to the VMFP equations with external magnetic field, follows the strategy from \cite{DipLion}. It starts from a regularized system that admits global solutions, for which the energy estimate provides the requisite uniform bounds. These bounds allow us to apply compactness arguments based on the compactness  in $L^1$ for hypoelliptic operators (see \cite{BouDol1995}). In this Section, we only show that the construction yields a solution which satisfies the local conservation law of momentum \eqref{equ:law_moment_meas} and the free-energy inequality \eqref{equ:ener_meas}. The existence result is deferred to Appendix, where we outline the main steps of the argument. 
\subsection{A priori estimates}
\label{Prio_est}
Consider a sequence $(f^\eps_\delta, E^\eps_\delta,B^\eps_\delta)_{\delta >0}$ of approximate solutions of \eqref{equ:VMFP-Scale}-\eqref{equ:Initial} (here $\eps$ is kept fixed), as constructed follows:
\begin{equation}
\label{equ:regu_VMFPEps}
\begin{split}
\left\{\begin{array}{l}
\eps \partial_t f^\eps_\delta + v\cdot\nabla_x f^\eps_\delta + \dfrac{q}{m}\left( K_\delta (E^\eps_\delta + v\wedge B^\eps_\delta) + v\wedge \dfrac{B_{\text{ext}}e}{\eps} \right)\cdot \nabla_v f^\eps_\delta = \dfrac{1}{\tau}\Divv\left(\sigma\nabla_v f^\eps_\delta + v f^\eps_\delta \right),\\
\mu_0 \epsilon_0 \eps \partial_t E^\eps_\delta = \cur B^\eps_\delta - \mu_0 K_\delta J^\eps_\delta,\quad \eps \partial_t B^\eps_\delta =- \cur E^\eps_\delta,\\
\epsilon_0 \Divx E^\eps_\delta = q (n^\eps_\delta -D),\quad \Divx B^\eps_\delta =0,\\
f^\eps_\delta(0,x,v) = f^\eps_0(x,v),\quad E^\eps_\delta(0,x) = E^\eps_0(x),\quad B^\eps_\delta(0,x) = B^\eps_0(x).
\end{array}\right.  
\end{split}
\end{equation}
With a slight abuse of notation, the operator $K_\delta$ is a convolution operator defined by
\[
K_\delta f(x) = \int_{\R^3} K_\delta (x-y)f(y)\mathrm{d}y,\quad f\in L^1_{\text{loc}}(\R^3),
\]
where $K_\delta$ is a standard mollifier satisfying:
\begin{align*}
K_\delta\in C^\infty_c(\R^3),\quad \text{supp}K_\delta \subset \bar{B}_\delta,\quad \int_{\R^3} K_\delta \mathrm{d}x =1,\\
K_\delta (x)\geq 0 \quad \text{and}\quad K_\delta(x) = K_\delta(-x)\quad \text{for}\,\,x\in\R^3.
\end{align*}
From $K_\delta(x) = K_\delta(-x)$, it follows that the operator $K_\delta$ is self-adjoint with respect to the $L^2$ scalar product.

Proposition \ref{regu_solu} in Appendix implies the following energy estimate
\[
\sup_{t\in[0,T],\delta>0} \left( \int_{\R^3}\int_{\R^3} |v|^2 f^\eps_\delta(t,x,v)\mathrm{d}x\mathrm{d}v + \int_{\R^3}\dfrac{\epsilon_0}{2m}|E^\eps_\delta|^2 + \dfrac{1}{2\mu_0 m}|B^\eps_\delta|^2\mathrm{d}x\right) <+\infty,
\]
as well as the estimate of $L^2$ norm
\[
\sup_{t\in[0,T],\delta>0}\int_{\R^3}\int_{\R^3}|f^\eps_\delta|^2(t,x,v)\mathrm{d}v\mathrm{d}x <+\infty.
\]
Then, up to extraction of a subsequence (still indexed by $\delta$)
\begin{align*}
f^\eps_\delta \rightharpoonup f^\eps,\quad w \star-L^\infty(0,T;L^2(\R^3\times\R^3),\\
(E^\eps_\delta, B^\eps_\delta)\rightharpoonup (E^\eps,B^\eps),\quad w\star - L^\infty(0,T;L^2(\R^3))^6.
\end{align*}
Therefore (after extraction eventually) there are symmetric nonnegative matrix-valued defect measures $\mu_E^\eps,\mu_B^\eps\in L^\infty(0,T;\mathcal{M}(\R^3))^9$, $\mu_{EB}^\eps\in L^\infty(0,T;\mathcal{M}(\R^3))^3$ such that for any $\varphi \in C^0([0,T]\times\R^3)$
\begin{equation*}
\begin{split}
\lim_{\delta\to 0} \int_0^T\int_{\R^3}(E^\eps_\delta\otimes E^\eps_\delta) \varphi(t,x)\mathrm{d}x\mathrm{d}t = \int_0^T\int_{\R^3}(E^\eps\otimes E^\eps) \varphi(t,x)\mathrm{d}x\mathrm{d}t + \int_0^T \int_{\R^3}\varphi(t,x)\mathrm{d}\mu_E^\eps,\\
\lim_{\delta\to 0} \int_0^T\int_{\R^3}(B^\eps_\delta\otimes B^\eps_\delta) \varphi(t,x)\mathrm{d}x\mathrm{d}t = \int_0^T\int_{\R^3}(B^\eps\otimes B^\eps) \varphi(t,x)\mathrm{d}x\mathrm{d}t + \int_0^T \int_{\R^3}\varphi(t,x)\mathrm{d}\mu_B^\eps,\\
\lim_{\delta\to 0} \int_0^T\int_{\R^3}(E^\eps_\delta\wedge B^\eps_\delta) \varphi(t,x)\mathrm{d}x\mathrm{d}t = \int_0^T\int_{\R^3}(E^\eps\wedge B^\eps) \varphi(t,x)\mathrm{d}x\mathrm{d}t + \int_0^T \int_{\R^3}\varphi(t,x)\mathrm{d}\mu_{EB}^\eps .
\end{split}
\end{equation*}
Moreover, from the Cauchy-Schwarz inequality, we deduce that
\[
\int_{\R^3} \mathrm{d}\mu_{EB}^\eps \leq \dfrac{1}{2}\int_{\R^3} tr(\mathrm{d}\mu_E^\eps) + tr(\mathrm{d}\mu_B^\eps).
\]
In addition, the energy estimate above implies that the sequence $(f^\eps_\delta |v|^2)_{\delta >0}$ is bounded in $L^\infty(0,T;\mathcal{M}(\R^3\times\R^3))$. Thus the sequence $\nu^\eps_\delta = \int_0^\infty r^2 f^\eps_\delta(t,x,r\sigma)r^2\mathrm{d}r$ of push-forwards of $f^\eps_\delta$ under the map $(t,x,v)\mapsto (t,x,\sigma = {v}/{|v|})$ is bounded in $L^\infty(0,T;\mathcal{M}(\R^3\times\Sp^2))$. Hence (after extraction eventually) there exists a nonnegative measure $m^\eps \in  L^\infty(0,T;\mathcal{M}(\R^3\times\Sp^2))$ such that
\begin{equation}
\label{equ:defect_meas}
\begin{split}
&\lim_{\delta\to 0}\int_0^T\int_{\R^3}\int_{\R^3} \psi\left(t,x,\dfrac{v}{|v|}\right)|v|^2 f^\eps_\delta(t,x,v)\mathrm{d}v\mathrm{d}x\mathrm{d}t\\
&= \int_0^T\int_{\R^3}\int_{\R^3} \psi\left(t,x,\dfrac{v}{|v|}\right)|v|^2 f^\eps(t,x,v)\mathrm{d}v\mathrm{d}x\mathrm{d}t  + \int_0^T\int_{\R^3}\int_{\Sp^2} \psi(t,x,\sigma)\mathrm{d}m^\eps(t,x,\sigma), 
\end{split}
\end{equation}
for any $\psi\in C^0([0,T]\times\R^3\times\Sp^2)$. Taking now $\psi(t,x,\sigma)=\theta(t,x)\sigma\otimes\sigma$ and $\psi(t,x,\sigma) = \frac{\theta(t,x)}{2}$, respectively in \eqref{equ:defect_meas}, we deduce that
\begin{equation*}
\begin{split}
\lim_{\delta\to 0}\int_{\R^3} (v\otimes v) f^\eps_\delta \mathrm{d}v &= \int_{\R^3} (v\otimes v) f^\eps\mathrm{d}v + \int_{\Sp^2}\sigma\otimes\sigma\mathrm{d}m^\eps,\\
\lim_{\delta\to 0}\int_{\R^3} \dfrac{|v|^2}{2}f^\eps_\delta \mathrm{d}v &= \int_{\R^3} \dfrac{|v|^2}{2} f^\eps\mathrm{d}v + \dfrac{1}{2}\int_{\Sp^2}\mathrm{d}m^\eps,
\end{split}
\end{equation*}
in the sense of distributions on $[0,T[\times\R^3$.
\subsection{The local conservation of momentum with the defect measures}
The local conservation of momentum holds rigorously in the sense of distributions on $[0,T[\times\R^3$ for approximate solutions $(f^\eps_\delta,E^\eps_\delta,B^\eps_\delta)$ (see Proposition \ref{regu_solu}):
\begin{equation}
\label{equ:moment_approx_sol}
\begin{split}
&\eps \partial_t \int_{\R^3} vf^\eps_\delta\mathrm{d}v + \Divx \int_{\R^3} v\otimes v f^\eps_\delta \mathrm{d}v - \dfrac{q}{m}(K_\delta E^\eps_\delta n^\eps_\delta + \int_{\R^3}v f^\eps_\delta\mathrm{d}v \wedge K_\delta B^\eps_\delta) \\
&- \dfrac{q}{m}\dfrac{1}{\eps}\int_{\R^3}vf^\eps_\delta\mathrm{d}v\wedge( B_{\text{ext}}e )
= -\dfrac{1}{\tau}\int_{\R^3}vf^\eps_\delta\mathrm{d}v.
\end{split}
\end{equation}
We need to compute $ \frac{q}{m}(E^\eps_\delta n^\eps_\delta + (K_\delta\int_{\R^3}vf^\eps_\delta\mathrm{d}v)\wedge B^\eps_\delta)$. From the Maxwell equations in system \eqref{equ:regu_VMFPEps}, we deduce that
\begin{equation*}
\begin{split}
q  E^\eps_\delta n^\eps_\delta &= E^\eps_\delta(\epsilon_0 \Divx E^\eps_\delta + qD),\\
 q \left(K_\delta\int_{\R^3}vf^\eps_\delta\mathrm{d}v\right) \wedge  B^\eps_\delta & = K_\delta J^\eps_\delta \wedge B^\eps_\delta= - \epsilon_0 \eps\partial_t E^\eps_\delta \wedge B^\eps_\delta + \dfrac{1}{\mu_0}\cur B^\eps_\delta \wedge  B^\eps_\delta.
\end{split}
\end{equation*}
Using the identity \eqref{equ:quad_field}, we easily deduce that
\begin{align*}
\epsilon_0 \Divx E^\eps_\delta E^\eps_\delta  &= \epsilon_0 E^\eps_\delta\wedge \cur E^\eps_\delta + \epsilon_0\mathcal{A}(E^\eps_\delta,E^\eps_\delta))\\
&= -\epsilon_0 \eps E^\eps_\delta \wedge \partial_t B^\eps_\delta + \epsilon_0 \mathcal{A}(E^\eps_\delta,E^\eps_\delta),
\end{align*}
and 
\begin{align*}
\cur B^\eps_\delta \wedge B^\eps_\delta = \mathcal{A}(B^\eps_\delta,B^\eps_\delta).
\end{align*}
Replacing in the previous identities, we obtain
\begin{align*}
\dfrac{q}{m}(E^\eps_\delta n^\eps_\delta + K_\delta\int_{\R^3}vf^\eps_\delta\mathrm{d}v\wedge B^\eps_\delta) = \dfrac{q}{m} D E^\eps_\delta + \dfrac{\epsilon_0}{m} \mathcal{A}(E^\eps_\delta,E^\eps_\delta) + \dfrac{\mu_0}{m}\mathcal{A}(B^\eps_\delta,B^\eps_\delta)   - \dfrac{\epsilon_0\eps}{m}\partial_t (E^\eps_\delta\wedge B^\eps_\delta).
\end{align*}
Inserting this identity into \eqref{equ:moment_approx_sol}, then taking limits in the  equation and using the defect measures introduced in Subsection \ref{Prio_est}, together with the fact that
\[
K_\delta E^\eps_\delta n^\eps_\delta  - E ^\eps_\delta n^\eps_\delta \to 0,\quad j^\eps_\delta \wedge K_\delta B^\eps_\delta - K_\delta j^\eps_\delta\wedge B^\eps_\delta \to 0\quad\text{in}\quad\mathcal{D}'((0,T)\times\R^3),
\]
 we conclude that the weak solution $(f^\eps,E^\eps,B^\eps)$ satisfies \eqref{equ:law_moment_meas}. 
\subsection{The free-energy estimate with the defect measures}
From Proposition \ref{regu_solu}, the approximate solutions $(f^\eps_\delta,E^\eps_\delta,B^\eps_\delta)$ satisfies the balance of free-energy. Taking the limit and using the defect measures, the limit solutions $(f^\eps,E^\eps,B^\eps)$ satisfy the free-energy estimate \eqref{equ:ener_meas}.

Now, we establish uniform bounds for the kinetic energy with the defect measures.
\begin{lemma}
\label{KinEne}
Let $T>0$. Assume that the initial data $(f^\eps_{0}, E^\eps_0,B^\eps_0)$ satisfy $f^\eps _{0} \geq 0$, $M_{0}:= \sup_{\eps >0} M^\eps _{0} < +\infty$, $U_{0} := \sup_{\eps >0} U^\eps _{0} < +\infty$, where for any $\eps >0$
\begin{align*}
M^\eps _{0} &:= \int_{\R^3}\int_{\R^3}{f^\eps _{0} (x,v)}\mathrm{d}v\mathrm{d}x,\\
U^\eps _{0} &:= \int_{\R^3}\int_{\R^3}{\dfrac{|v|^2}{2}f^\eps _{0} (x,v)}\mathrm{d}v\mathrm{d}x + \dfrac{\epsilon_0}{2m}\int_{\R^3}{|E^\eps_0|^2}\mathrm{d}x + \dfrac{1}{2\mu_0 m}\int_{\R^3}|B^\eps_0|^2\mathrm{d}x.
\end{align*}
We assume that $(f^\eps,E^\eps,B^\eps)_{\eps>0}$ are weak solutions of \eqref{equ:VMFP-Scale}-\eqref{equ:Initial}. Then we have 
\begin{equation*}
\begin{split}
\eps \sup_{0\leq t\leq T} \left\{ \int_{\R^3}\int_{\R^3}{\dfrac{|v|^2}{2}f^\eps}\mathrm{d}v\mathrm{d}x + \int_{\R^3}\left( \dfrac{\epsilon_0}{2m}{|E^\eps|^2} + \dfrac{1}{2\mu_0 m}|B^\eps|^2\right)\mathrm{d}x\right\} \leq \eps U_{0} + \dfrac{3\sigma}{\tau}TM_{0},
\end{split}
\end{equation*}
and
\[
\dfrac{1}{\tau} \int_0^T\int_{\R^3}\int_{\R^3}{|v|^2 f^\eps(t,x,v)}\mathrm{d}v\mathrm{d}x\mathrm{d}t \leq \eps U_{0} + \dfrac{3\sigma}{\tau}TM_{0}.
\]
\end{lemma}
\begin{proof}
We first establish the results for approximate solutions $(f^\eps_\delta,E^\eps_\delta,B^\eps_\delta)$ with initial data $(f^\eps_0,E^\eps_0,B^\eps_0)$.
Multiplying \eqref{equ:VMFP-Scale} by $\frac{|v|^2}{2}$ and integrating with respect to $(x,v)\in\R^3\times\R^3$ then using the Maxwell equations \eqref{equ:MaxwellEps} yields the following the energy identity:
\[
\eps \dfrac{\mathrm{d}}{\mathrm{d}t} \left\{ \int_{\R^3}\int_{\R^3}{\dfrac{|v|^2}{2}f^\eps_\delta}\mathrm{d}v\mathrm{d}x + \int_{\R^3}\left({\dfrac{\epsilon_0}{2m}|E^\eps_\delta|^2} + \dfrac{1}{2\mu_0 m}|B^\eps_\delta|^2 \right)\mathrm{d}x \right\} = \dfrac{3\sigma}{\tau}M^\eps _{0} - \dfrac{1}{\tau} \int_{\R^3}\int_{\R^3}{|v|^2 f^\eps_\delta}\mathrm{d}v\mathrm{d}x.
\]
A detail proof of this identity is provided in Lemma \ref{regu_loc_ener} in Appendix. Therefore we obtain
\begin{align*}
&\eps  \left\{ \int_{\R^3}\int_{\R^3}{\dfrac{|v|^2}{2}f^\eps_\delta}\mathrm{d}v\mathrm{d}x + \int_{\R^3}\left({\dfrac{\epsilon_0}{2m}|E^\eps_\delta|^2} + \dfrac{1}{2\mu_0 m}|B^\eps_\delta|^2 \right)\mathrm{d}x \right\} + \dfrac{1}{\tau} \int_0^t\int_{\R^3}\int_{\R^3}{|v|^2 f^\eps_\delta}\mathrm{d}v\mathrm{d}x\mathrm{d}s \\
&= \eps U^\eps _{0}+  \dfrac{3\sigma}{\tau} t M^\eps _{0}.
\end{align*}
Taking limits and using the defect measures in Subsection \ref{Prio_est}, we deduce that the weak solution $(f^\eps,E^\eps,B^\eps)$ satisfies the following equation:
\begin{align*}
&\eps  \left\{ \int_{\R^3}\int_{\R^3}{\dfrac{|v|^2}{2}f^\eps}\mathrm{d}v\mathrm{d}x + \int_{\R^3}\left({\dfrac{\epsilon_0}{2m}|E^\eps|^2} + \dfrac{1}{2\mu_0 m}|B^\eps|^2 \right)\mathrm{d}x \right\} + \dfrac{1}{\tau} \int_0^t\int_{\R^3}\int_{\R^3}{|v|^2 f^\eps}\mathrm{d}v\mathrm{d}x\mathrm{d}s \\
&\quad +\dfrac{1}{2}\int_{\R^3}\int_{\R^3}\mathrm{d}m^\eps + \dfrac{1}{2}\int_{\R^3}\mathrm{d}\left( \dfrac{\epsilon_0}{m}tr(\mu^\eps_E) + \dfrac{1}{\mu_0 m} tr(\mu^\eps_B) \right) + \dfrac{1}{\tau}\int_0^t\int_{\R^3}\int_{\R^3}\mathrm{d}m^\eps\\
&= \eps U^\eps _{0}+  \dfrac{3\sigma}{\tau} t M^\eps _{0},\,\,t\in[0,T],
\end{align*}
which implies the desired results.
\end{proof}
\section{Analysis of the limit system}
\label{lim_mod_2d}
We start with the analysis of the system \eqref{equ:lim_mod_2dx}. Note that this system can be split into two problems. First solves for $(n,E)$
\begin{equation}
\label{equ:lim_mod_2d_bis1}
\left\{\begin{array}{l}
\partial_t n (t,x_\perp) + \Divx\left[\dfrac{n e}{\omega_c}\wedge \left( \sigma \dfrac{\nabla_x n}{n} - \dfrac{q}{m}E\right)\right] =0,\,\,(t,x=(x_\perp,x_3))\in(0,T)\times\R^3,\\
\epsilon_0 \Divx E = q(n -D),\quad \cur E =0,\quad E_3 =0,\,\,\,\, E=(E_1,E_2,E_3),\\
n(0,x_\perp)=n_0(x_\perp),\,\,
\end{array}\right.
\end{equation}
and secondly find $B_1$ solution of 
\begin{equation}
\label{equ:lim_mod_2d_bis2}
\left\{\begin{array}{l}
 \cur B_1(t,x_\perp)  =\mu_0\epsilon_0\partial_t E(t,x_\perp) + \mu_0 q \dfrac{n e}{\omega_c}\wedge \left( \sigma \dfrac{\nabla_x n}{n} - \dfrac{q}{m}E \right),\\
\Divx B_1 = 0.
\end{array}\right.
\end{equation}
We give here an existence result for \eqref{equ:lim_mod_2d_bis1} which is a direct consequence of the existence result obtained in \cite{BosTuan}. The system \eqref{equ:lim_mod_2d_bis1} can be recast as:
\begin{equation*}
\begin{split}
\left\{\begin{array}{l}
\partial_t n (t,x_\perp) + \mathrm{div}_{x_\perp}\left[n \left( \begin{matrix}
\dfrac{E_2}{B_{\text{ext}}} - \dfrac{\sigma}{\omega_c}\dfrac{\partial_{x_2}n}{n}\\
-\dfrac{E_1}{B_{\text{ext}}} + \dfrac{\sigma}{\omega_c}\dfrac{\partial_{x_1}n}{n}
\end{matrix}
\right)\right] =0,\quad (t,x_\perp)\in (0,T)\times\R^2,\\
\epsilon_0 (\partial_{x_1}E_1 + \partial_{x_2}E_2) = q(n -D),\\
\partial_{x_1}E_2 - \partial_{x_2}E_1 =0,\quad E_3 =0,\\
n(0,x_\perp) = n_0(x_\perp).
\end{array}\right.
\end{split}
\end{equation*}
\begin{pro}
\label{existsol_lim_mod1}
Let $T>0$. Let $B_{\text{ext}}\in C^2_b(\R^2)$ be a smooth magnetic field such that $\inf_{x_\perp\in\R^2} B_{\text{ext}}(x_\perp) = B_0 >0$ and the fixed background density $D$ verifies that $|x_\perp| D\in L^1(\R^2), D\in L^1(\R^2)\cap W^{1,\infty}(\R^2)$. Assume that the initial condition $n_0(x_\perp)$ satisfies
\begin{align*}
n_0 \geq 0,\quad |x_\perp|n_0 \in L^1(\R^2),\quad n_0\in W^{1,1}(\R^2)\cap W^{1,\infty}(\R^2),\\
\int_{\R^2} n_0(x_\perp) \mathrm{d}x_\perp = \int_{\R^2} D(x_\perp)\mathrm{d}x_\perp.
\end{align*}
Then there is a unique regular solution $n(t,x_\perp)\geq 0$ on $[0,T]\times \R^2$ of \eqref{equ:lim_mod_2d_bis1} that has the following properties
\begin{align*}
\int_{\R^2}n(t,x_\perp)\mathrm{d}x_\perp = \int_{\R^2} D(x_\perp)\mathrm{d}x_\perp,\\
n \in W^{1,\infty}(0,T;L^1(\R^2))\cap W^{1,\infty}((0,T)\times\R^2),\quad |x_\perp|n \in L^\infty(0,T;L^1(\R^2)),\\
E\in W^{1,\infty}((0,T)\times \R^2),\quad E,\,\,\partial_t E\in L^\infty(0,T;L^2(\R^2)).
\end{align*}
Moreover, if we assume that $B_{\text{ext}}\in C^3_b(\R^2)$, $D\in W^{1,1}(\R^2)\cap W^{2,\infty}(\R^2)$ and the initial condition  $n_0(x) $ satisfies
\begin{align*}
n_0 \in W^{2,1}(\R^2)\cap W^{2,\infty}(\R^2),\quad \ln n_0 \in W^{2,\infty}(\R^2)
\end{align*}
then we have 
\[
k[n]  \in W^{1,\infty}((0,T)\times\R^2),\quad k[n] = \sigma\dfrac{\nabla_x n}{n} - \dfrac{q}{m}E.
\]
\end{pro}

Once we find $(n,E)$ it is easy to solve the equation \eqref{equ:lim_mod_2d_bis2}
\begin{pro}
\label{existsol_lim_mod2}
Under the hypotheses of Proposition \ref{existsol_lim_mod1}, the system \eqref{equ:lim_mod_2d_bis2} has a unique solution $B_1$ satisfying $B_1\in L^\infty(0,T;H^1(\R^2))\cap L^\infty(0,T;W^{1,p}(\R^2))$ with $p>3$, and $\partial_t B_1 \in L^\infty(0,T;L^2(\R^2))$. In particular, $B_1\in L^\infty((0,T)\times\R^2)$.
\end{pro}
\begin{proof}
Observe from Proposition \ref{existsol_lim_mod1} that we have $\Divx (\mu_0 \epsilon_0 \partial_t E + \mu_0 q \frac{ne}{\omega_c}\wedge (\sigma\frac{\nabla_x n}{n}- \frac{q}{m}E))=0$ and that $\mu_0 \epsilon_0 \partial_t E + \mu_0 q \frac{ne}{\omega_c}\wedge (\sigma\frac{\nabla_x n}{n}- \frac{q}{m}E)\in L^\infty(0,T;L^2(\R^2)\cap L^\infty(\R^2))$. Therefore there is a unique $B_1\in L^\infty(0,T;H^1(\R^2)\cap W^{1,p}(\R^2))$  with $p> 3$ satisfying \eqref{equ:lim_mod_2d_bis2}. In particular, since $p>3$, we have $B_1\in L^\infty((0,T)\times\R^2)$. In order to estimate $\partial_t B_1\in L^\infty(0,T;L^2(\R^2))$ it is sufficient to estimate $\partial_t(\partial_t E + n e\wedge(\frac{\nabla_x n}{n} -E))\in L^\infty(0,T;H^{-1}(\R^2))$. We have
\begin{align*}
\epsilon_0  \mathrm{div}_{x_\perp}(\partial^2_t E_1,\partial^2_t E_2) = q \partial^2_t (n-D) = -q \partial_t \mathrm{div}_{x_1,x_2}(n A[n]),\quad A[n] =\left( \begin{matrix}
\dfrac{E_2}{B_{\text{ext}}} - \dfrac{\sigma}{\omega_c}\dfrac{\partial_{x_2}n}{n}\\
-\dfrac{E_1}{B_{\text{ext}}} + \dfrac{\sigma}{\omega_c}\dfrac{\partial_{x_1}n}{n}
\end{matrix}\right),
\end{align*}
and thus
\begin{align*}
\|\partial^2_t (E_1,E_2) \|_{L^\infty(0,T;H^{-1})} &\leq C\|\partial_t (nA[n])\|_{L^\infty(0,T;H^{-1})}\\
&\leq C (\|\partial_t n\|_{L^\infty} \|(E_1,E_2)\|_{L^\infty(0,T;L^2)}  + \|n\|_{L^\infty}\|\partial_t E\|_{L^\infty(0,T;L^2)})\\
&\quad + C \|\partial_t \nabla_x n\|_{L^\infty(0,T;H^{-1}(\R^2))}< +\infty,
\end{align*}
since $\|\partial_t \nabla_x n\|_{L^\infty(0,T;H^{-1}(\R^2))} \leq \|\partial_t n\|_{L^\infty(0,T;L^2(\R^2))}$.
\end{proof}
\begin{remark}
We emphasize that the regular solution of the limit model \eqref{equ:lim_mod_2dx}, obtained in this section, rigorously satisfies the properties stated in Proposition \ref{BalLiMod}.
\end{remark}

\section{Convergence result}
\label{Conve}
We now focus on the asymptotic behavior, as $\eps \searrow 0$ of the family of weak solutions $(f^\eps, E^\eps,B^\eps)_{\eps>0}$ of the VPFP system \eqref{equ:VMFP-Scale2dx}, \eqref{equ:MaxwellEps}, and \eqref{equ:Initial} and we rigorously establish the connection to the fluid model \eqref{equ:lim_mod_2dx}. \\
We are looking a model for the concentration $n^\eps = n[f^\eps] =\int_{\R^3}{f^\eps}\mathrm{d}v$, similar to equation \eqref{equ:lim_mod_2dx} of the limit concentration $n$ and we perform the balance of the relative entropy between $n^\eps$ and $n$. As usual, these computations require the regular of the solution for the limit model. We justify the asymptotic behavior of  $(f^\eps, E^\eps,B^\eps)_{\eps>0}$ when $\eps \searrow 0$, provided that there is a regular solution $(n, E, B_1)$ for the fluid model \eqref{equ:lim_mod_2dx}. 

Using the local conservation law of the momentum \eqref{equ:law_moment_meas},
we can express the orthogonal component of $\frac{1}{\eps}\int_{\R^3}v f^\eps\mathrm{d}v$ as follows
\begin{align*}
\int_{\R^3}\dfrac{(v-(v\cdot e)e)}{\eps}f^\eps\mathrm{d}v 
&= \dfrac{ e}{\omega_c}\wedge \left(\sigma \nabla_x n^\eps - \dfrac{\epsilon_0}{m}\mathcal{A}(E^\eps,E^\eps) - \dfrac{q}{m}D E^\eps \right) \\
&\quad -\dfrac{e}{\omega_c}\wedge \left(\dfrac{1}{m\mu_0}\mathcal{A}(B^\eps,B^\eps)\right)+ \dfrac{e}{\omega_c}\wedge\left( \dfrac{\epsilon_0\eps}{m}\partial_t (E^\eps\wedge B^\eps)\right)\\
&\quad + \dfrac{e}{\omega_c}\wedge R^\eps + \dfrac{e}{\omega_c}\wedge F^\eps,
\end{align*}
where we denote the following group of defect terms
\begin{equation*}
\begin{split}
R^\eps &=  \dfrac{e}{\omega_c}\wedge\left[  
 - \dfrac{\epsilon_0}{m}\Divx\left( \mu^\eps_E - \dfrac{1}{2}tr(\mu^\eps_E)I_3 \right) - \dfrac{1}{m\mu_0}\Divx\left(\mu^\eps_B - \dfrac{1}{2}tr(\mu^\eps_B)I_3 \right)\right] \\
&\quad + \dfrac{e}{\omega_c}\wedge \left( \dfrac{\epsilon_0\eps}{m}\partial_t \mu^\eps_{EB}\right)  + \dfrac{e}{\omega_c}\wedge \Divx\int_{\Sp^2}\sigma\otimes\sigma \mathrm{d}m^\eps,
\end{split}
\end{equation*}
and the remaining terms
\[
F^\eps =  \Divx\int_{\R^3}{(\sigma \nabla_v f^\eps + v f^\eps)\otimes v}\mathrm{d}v + \eps\partial_t \int_{\R^3}v f^\eps\mathrm{d}v + \dfrac{1}{\tau}\int_{\R^3}vf^\eps\mathrm{d}v,
\]
and in the above computation, we have used that $\Divx\intvt{\sigma \nabla_v f^\eps \otimes v} = - \sigma \nabla_x n^\eps$.\\
Observe that 
\begin{equation}
\label{equ:current}
\begin{split}
 \int_{\R^3}\dfrac{v}{\eps}f^\eps \mathrm{d}v &=  \int_{\R^3}\dfrac{v - (v\cdot e)e}{\eps}f^\eps\mathrm{d}v +  \int_{\R^3} \dfrac{(v\cdot e)e}{\eps} f^\eps \mathrm{d}v\\
&= \dfrac{e}{\omega_c}\wedge \left(\sigma \nabla_x n^\eps - \dfrac{\epsilon_0}{m}\mathcal{A}(E^\eps,E^\eps)-\dfrac{q}{m}DE^\eps\right) \\
&\quad -\dfrac{e}{\omega_c}\wedge \left(\dfrac{1}{m\mu_0}\mathcal{A}(B^\eps,B^\eps)\right)+ \dfrac{e}{\omega_c}\wedge\left( \dfrac{\epsilon_0\eps}{m}\partial_t (E^\eps\wedge B^\eps)\right)\\
&\quad+  \dfrac{e}{\omega_c}\wedge R^\eps+ \dfrac{e}{\omega_c}\wedge F^\eps + \int_{\R^3}\dfrac{(v\cdot e)e}{\eps}f^\eps\mathrm{d}v.
\end{split}
\end{equation}
Finally, using \eqref{equ:current} and thanks to the local conversation of mass \eqref{equ:law_charge_meas}, we obtain a similar model for $n^\eps$, as in the first equation of \eqref{equ:lim_mod_2dx} as follows:
\begin{equation}
\label{equ:EquDensityEps}
\begin{split}
\partial_t n^\eps &+ \Divx\left( \dfrac {e}{\omega_c}\wedge  \left( \sigma\nabla_x n^\eps - \dfrac{\epsilon_0}{m}\mathcal{A}(E^\eps,E^\eps) -\dfrac{q}{m}DE^\eps \right) \right)\\
&+\Divx\left( -\dfrac{e}{\omega_c}\wedge \left(\dfrac{1}{m\mu_0}\mathcal{A}(B^\eps,B^\eps)\right)+ \dfrac{e}{\omega_c}\wedge\left( \dfrac{\epsilon_0\eps}{m}\partial_t (E^\eps\wedge B^\eps)\right)\right)\\
&+ \Divx\left( \dfrac{e}{\omega_c}\wedge R^\eps\right)+ \Divx\left( \dfrac{e}{\omega_c}\wedge F^\eps\right) =0,
\end{split}
\end{equation}
since $e = (0,0,1)^t$ hence \[\Divx \int_{\R^3} (v\cdot e)e f^\eps(t,x_\perp,v) \mathrm{d}v= \partial_{x_3}\int_{\R^3}v_3 f^\eps(t,x_\perp,v)\mathrm{d}v =0.\]

We intend to estimate the modulated energy of $n^\eps$ with respect to $n$, given by \eqref{equ:mod_energy}, by expressing $\mathcal{E}[n^\eps|n]$ as follows:
\begin{align}
\label{equ:EntropyDens}
\begin{split}
\mathcal{E}[n^\eps|n] 
&= \int_{\R^2}{(\sigma n^\eps \ln n^\eps +\dfrac{\epsilon_0}{2m} |E^\eps|^2 + \dfrac{1}{2\mu_0 m}|B^\eps|^2)}\mathrm{d}x_\perp\\
&\quad - \int_{\R^2}{(\sigma n \ln n +\dfrac{\epsilon_0}{2m} |E|^2)}\mathrm{d}x_\perp + \dfrac{\eps^2}{2\mu_0 m}\int_{\R^2}|B_1|^2\mathrm{d}x_\perp \\
&\quad - \int_{\R^2}\sigma(1+\ln n)(n^\eps -n)\mathrm{d}x_\perp - \dfrac{\epsilon_0}{m}\int_{\R^2}E\cdot(E^\eps-E)\mathrm{d}x_\perp -\dfrac{\eps}{\mu_0 m}\int_{\R^2}B^\eps\cdot B_1\mathrm{d}x_\perp\\
&= \mathcal{E}[n^\eps] - \mathcal{E}[n] + \dfrac{\eps^2}{2\mu_0 m}\int_{\R^2}|B_1|^2\mathrm{d}x_\perp \\
&\quad - \int_{\R^2}\sigma(1+\ln n)(n^\eps -n)\mathrm{d}x_\perp - \dfrac{\epsilon_0}{m}\int_{\R^2}E\cdot(E^\eps-E)\mathrm{d}x_\perp -\dfrac{\eps}{\mu_0 m}\int_{\R^2}B^\eps\cdot B_1\mathrm{d}x_\perp,
\end{split}
\end{align}
where we denote
\[
\mathcal{E}[n^\eps] = \int_{\R^2}{(\sigma n^\eps \ln n^\eps +\dfrac{\epsilon_0}{2m} |E^\eps|^2 + \dfrac{1}{2\mu_0 m}|B^\eps|^2)}\mathrm{d}x_\perp,
\]
and
\[
\mathcal{E}[n] = \int_{\R^2}{(\sigma n \ln n +\dfrac{\epsilon_0}{2m} |E|^2)}\mathrm{d}x_\perp.
\]
We introduce as well the entropy relative of $f^\eps$ with respect to $n^\eps M$, given by
\begin{align*}
&\sigma\int_{\R^2}\int_{\R^3}{n^\eps(t) M h\left(\dfrac{f^\eps(t)}{n^\eps(t) M} \right)}\mathrm{d}v\mathrm{d}x_\perp \\
&= \sigma \int_{\R^2}\int_{\R^3}\left(f^\eps \ln f^\eps - f^\eps \ln n^\eps + f^\eps \ln (2\pi\sigma)^{3/2}+ f^\eps \frac{|v|^2}{2\sigma}\right)\mathrm{d}v\mathrm{d}x_\perp\\
&= \int_{\R^2}\int_{\R^3}\left(\sigma \ln f^\eps + \dfrac{|v|^2}{2}\right)f^\eps\mathrm{d}v\mathrm{d}x_\perp + \dfrac{\epsilon_0}{2m}\int_{\R^2}{|E^\eps|^2}\mathrm{d}x_\perp + \dfrac{1}{2\mu_0 m}\int_{\R^2}|B^\eps|^2\mathrm{d}x_\perp \\
&\quad- \left( \int_{\R^2}{\sigma n^\eps \ln n^\eps}\mathrm{d}x_\perp + \dfrac{\epsilon_0}{2m}\int_{\R^2}{|E^\eps|^2}\mathrm{d}x_\perp + \dfrac{1}{2\mu_0 m}\int_{\R^2}|B^\eps|^2\mathrm{d}x_\perp\right)\\
&\quad + \sigma \ln(2\pi\sigma)^{3/2}\int_{\R^2}\int_{\R^3}{f^\eps}\mathrm{d}v\mathrm{d}x_\perp\\
&= \int_{\R^2}\int_{\R^3}\left(\sigma \ln f^\eps + \dfrac{|v|^2}{2}\right)f^\eps\mathrm{d}v\mathrm{d}x_\perp + \dfrac{\epsilon_0}{2m}\int_{\R^2}{|E^\eps|^2}\mathrm{d}x_\perp + \dfrac{1}{2\mu_0 m}\int_{\R^2}|B^\eps|^2\mathrm{d}x_\perp \\
&\quad - \mathcal{E}[n^\eps(t)] + \sigma \ln(2\pi\sigma)^{3/2}\int_{\R^2}\int_{\R^3}{f^\eps}\mathrm{d}v\mathrm{d}x_\perp.
\end{align*}
Thanks to the free-energy estimate \eqref{equ:ener_meas} and mass conservation of equation \eqref{equ:VMFP-Scale2dx}, we obtain
\begin{equation}
\label{equ:BalanEnerDens}
\begin{split}
 &\mathcal{E}[n^\eps(t)]  +  \sigma\int_{\R^2}\int_{\R^3}{n^\eps(t) M h\left(\dfrac{f^\eps(t)}{n^\eps(t) M} \right)}\mathrm{d}v\mathrm{d}x_\perp \\
 &\quad+ \dfrac{1}{2}\int_{\R^2}\int_{\Sp^2}\mathrm{d}m^\eps + \dfrac{1}{2}\int_{\R^2} \mathrm{d}\left( \dfrac{\epsilon_0}{m}tr(\mu^\eps_E) + \dfrac{\mu_0}{m}tr(\mu^\eps_B)\right)
\\
&\quad + \dfrac{1}{\eps\tau}\int_0^t\int_{\R^2}\int_{\R^3}{\dfrac{|\sigma \nabla_v f^\eps + v f^\eps|^2}{f^\eps}}\mathrm{d}v\mathrm{d}x_\perp\mathrm{d}s\\
&\leq \mathcal{E}[n^\eps(0)] + \sigma\int_{\R^2}\int_{\R^3}{n^\eps(0) M h\left(\dfrac{f^\eps(0)}{n^\eps(0) M} \right)}\mathrm{d}v\mathrm{d}x_\perp.
\end{split}
\end{equation}
Thanks to Proposition \ref{BalLiMod} for the limit model \eqref{equ:lim_mod_2dx} and by combining \eqref{equ:EntropyDens}, \eqref{equ:BalanEnerDens}, we obtain
\begin{equation}
\label{BalModEnerDens}
\begin{split}
&\mathcal{E}[n^\eps(t)|n(t)] +  \sigma\int_{\R^2}\int_{\R^3}{n^\eps(t) M h\left(\dfrac{f^\eps(t)}{n^\eps(t) M} \right)}\mathrm{d}v\mathrm{d}x_\perp + \dfrac{1}{2}\int_{\R^2}\int_{\Sp^2}\mathrm{d}m^\eps \\
& \quad  + \dfrac{1}{2}\int_{\R^2} \mathrm{d}\left( \dfrac{\epsilon_0}{m}tr(\mu^\eps_E) + \dfrac{\mu_0}{m}tr(\mu^\eps_B)\right) + \dfrac{1}{\eps\tau}\int_0^t\int_{\R^2}\int_{\R^3}{\dfrac{|\sigma \nabla_v f^\eps + v f^\eps|^2}{f^\eps}}\mathrm{d}v\mathrm{d}x_\perp\mathrm{d}s\\
&\leq\mathcal{E}[n^\eps(0)|n(0)] +  \sigma\int_{\R^2}\int_{\R^3}{n^\eps(0) M h\left(\dfrac{f^\eps(0)}{n^\eps(0) M} \right)}\mathrm{d}v\mathrm{d}x_\perp \\
&\quad + \dfrac{\eps^2}{2\mu_0 m}\left(\int_{\R^2}|B_1(t)|^2\mathrm{d}x_\perp - \int_{\R^2}|B_1(0)|^2\mathrm{d}x_\perp \right) - \int_0^t \dfrac{\mathrm{d}}{\mathrm{d}s}\int_{\R^2}\sigma(1+\ln n)(n^\eps -n)\mathrm{d}x_\perp\mathrm{d}s \\
&\quad - \dfrac{\epsilon_0}{m}\int_0^t \dfrac{\mathrm{d}}{\mathrm{d}s}\int_{\R^2} E \cdot (E^\eps -E)\mathrm{d}x_\perp\mathrm{d}s - \dfrac{\eps}{\mu_0 m}\int_0^t\dfrac{\mathrm{d}}{\mathrm{d}s}\int_{\R^2}B^\eps\cdot B_1\mathrm{d}x_\perp\mathrm{d}s.
\end{split}
\end{equation}

The next task is to evaluate the time derivative of  the three last contributions in \eqref{BalModEnerDens}.
\begin{pro}
\label{TimeDeriv}
Let $T>0$. With the notations introduced above, we have the following equality in $\mathcal{D}'([0,T[)$:
\begin{align*}
&-\dfrac{\mathrm{d}}{\mathrm{d}t}\int_{\R^2}\sigma(1+\ln n)(n^\eps -n)\mathrm{d}x_\perp
 - \dfrac{\epsilon_0}{m} \dfrac{\mathrm{d}}{\mathrm{d}t}\int_{\R^2} E \cdot (E^\eps -E)\mathrm{d}x_\perp - \dfrac{\eps}{\mu_0 m}\dfrac{\mathrm{d}}{\mathrm{d}t}\int_{\R^2}B^\eps\cdot B_1\mathrm{d}x_\perp\\
&=  - \dfrac{\eps}{\mu_0 m}\int_{\R^2} B^\eps\cdot \partial_t B_1\mathrm{d}x_\perp + \dfrac{q}{m}\int_{\R^2}W[n]\cdot n(E^\eps -E)\mathrm{d}x_\perp + \sigma \int_{\R^2}\Divx (n W[n])\dfrac{n^\eps -n}{n}\mathrm{d}x_\perp\\
&\quad + \int_{\R^2} W[n] \cdot \left(\sigma \nabla_x n^\eps -\dfrac{\epsilon_0}{m}\mathcal{A}(E^\eps,E^\eps)-\dfrac{q}{m}DE^\eps\right)\mathrm{d}x_\perp \\
&\quad -  \int_{\R^2}W[n]\cdot \dfrac{1}{\mu_0 m}\mathcal{A}(B^\eps,B^\eps)\mathrm{d}x_\perp + \int_{\R^2}W[n] \cdot \left(\dfrac{\epsilon_0\eps}{m}\partial_t(E^\eps\wedge B^\eps)\right)\mathrm{d}x_\perp\\
&\quad + \int_{\R^2} W[n]\cdot  R^\eps\mathrm{d}x_\perp +  \int_{\R^2}W [n]\cdot F^\eps\mathrm{d}x_\perp .
\end{align*}
where we denote $W[n]= \dfrac{e}{\omega_c}\wedge k[n]$, with $k[n] = \sigma\dfrac{\nabla_x n}{n} - \dfrac{q}{m}E$.
\end{pro}
\begin{proof}
First, we have the following equality in $\mathcal{D}'([0,T[)$:
\begin{equation}
\label{equ:time_deri_mag}
-\dfrac{\eps}{\mu_0 m} \dfrac{\mathrm{d}}{\mathrm{d}t}\int_{\R^2}B^\eps\cdot B_1\mathrm{d}x_\perp = -\dfrac{\eps}{\mu_0 m}\int_{\R^2}\partial_t B^\eps\cdot B_1\mathrm{d}x_\perp  -\dfrac{\eps}{\mu_0 m}\int_{\R^2}B^\eps\cdot \partial_t B_1\mathrm{d}x_\perp.
\end{equation}

Next, by straightforward computations and using the equations for the densities $n^\eps(t,x_\perp)$ in \eqref{equ:EquDensityEps} and $n(t,x_\perp)$ in \eqref{equ:lim_mod_2dx} respectively, we have in $\mathcal{D}'([0,T[)$ that
\begin{equation}
\label{equ:time_deri_dens}
\begin{split}
&-\dfrac{\mathrm{d}}{\mathrm{d} t}\int_{\R^2}{\sigma(1+\ln n)(n^\eps-n)}\mathrm{d}x_\perp = -\sigma \int_{\R^2} \partial_t n  \dfrac{n^\eps - n}{n} \mathrm{d}x_\perp - \sigma \int_{\R^2}(1+\ln n)(\partial_t n^\eps -\partial_t n)\mathrm{d}x_\perp\\
&= \sigma \int_{\R^2}\Divx \left( \dfrac{n e}{\omega_c}\wedge k[n]\right)\dfrac{n^\eps -n}{n}\mathrm{d}x_\perp\\
&\quad +  \int_{\R^2}\sigma \dfrac{\nabla_x n}{n}\cdot \left[ \dfrac{n e}{\omega_c}\wedge k[n] - \dfrac{e}{\omega_c}\wedge\left( \sigma\nabla_x n^\eps - \dfrac{\epsilon_0}{m} \mathcal{A}(E^\eps,E^\eps) -\dfrac{q}{m}DE^\eps\right)\right]\mathrm{d}x_\perp\\
&\quad + \dfrac{1}{\mu_0 m}\int_{\R^2} \sigma\dfrac{\nabla_x n}{n}\cdot \left( \dfrac{e}{\omega_c}\wedge \mathcal{A}(B^\eps,B^\eps)\right)\mathrm{d}x_\perp - \dfrac{\epsilon_0 \eps}{m}\int_{\R^2}\sigma\dfrac{\nabla_x n}{n} \cdot \left(\dfrac{e}{\omega_c}\wedge \partial_t(E^\eps\wedge B^\eps) \right)\mathrm{d}x_\perp \\
&\quad -  \int_{\R^2}\sigma\dfrac{\nabla_x n}{n}\cdot \left( \dfrac{e}{\omega_c}\wedge R^\eps\right)\mathrm{d}x_\perp -\int_{\R^2}\sigma\dfrac{\nabla_x n}{n}\cdot \left(\dfrac{e}{\omega_c}\wedge F^\eps \right)\mathrm{d}x_\perp.
\end{split}
\end{equation}

Finally, we estimate
\begin{equation}
\label{equ:time_deri_E}
\begin{split}
- \dfrac{\epsilon_0}{m}\dfrac{\mathrm{d}}{\mathrm{d}t}\int_{\R^2} E\cdot (E^\eps -E)\mathrm{d}x_\perp 
= - \dfrac{\epsilon_0}{m}\int_{\R^2} \partial_t E\cdot (E^\eps -E)\mathrm{d}x_\perp - \dfrac{\epsilon_0}{m}\int_{\R^2} E\cdot (\partial_t E^\eps - \partial_t E)\mathrm{d}x_\perp.
\end{split}
\end{equation}
Using the Maxwell equations \eqref{equ:MaxwellEps}, we have
\begin{align*}
\int_{\R^2}\cur B_1\cdot (E^\eps -E)\mathrm{d}x_\perp &= \int_{\R^2} B_1\cdot \cur(E^\eps -E)\mathrm{d}x_\perp\\
&=-\eps\int_{\R^2} B_1\cdot \partial_t B^\eps \mathrm{d}x_\perp,\quad \text{since}\,\,\,\, \cur E =0.
\end{align*}
Therefore, using the last equation in \eqref{equ:lim_mod_2dx}, we deduce that the first term in \eqref{equ:time_deri_E} can be written as follows:
\begin{equation}
\label{equ:time_deri_E1}
- \dfrac{\epsilon_0}{m}\int_{\R^2} \partial_t E\cdot (E^\eps -E)\mathrm{d}x_\perp = \dfrac{\eps}{m\mu_0}\int_{\R^2} B_1\cdot \partial_t B^\eps \mathrm{d}x_\perp + \dfrac{q}{m}\int_{\R^2}\left(\dfrac{e}{\omega_c}\wedge k[n]\right)\cdot n(E^\eps -E)\mathrm{d}x_\perp.
\end{equation}
On the other hand, from $\epsilon_0 \partial_t E^\eps = \dfrac{1}{\mu_0 \eps}\cur B^\eps - \dfrac{q}{\eps}\int_{\R^3}vf^\eps\mathrm{d}v$ in the Maxwell equations \eqref{equ:MaxwellEps} and using \eqref{equ:current}, we deduce that
\begin{align*}
\epsilon_0 \partial_t E^\eps &= \dfrac{1}{\mu_0 \eps}\cur B^\eps - \dfrac{qe}{\omega_c}\wedge \left[ \sigma \nabla_x n^\eps - \dfrac{\epsilon_0}{m}\mathcal{A}(E^\eps,E^\eps) - \dfrac{q}{m}DE^\eps\right] \\
&\quad + \dfrac{qe}{\omega_c}\wedge \dfrac{1}{m\mu_0}\mathcal{A}(B^\eps,B^\eps) - \dfrac{qe}{\omega_c}\wedge \left(\dfrac{\epsilon_0\eps}{m}\partial_t(E^\eps\wedge B^\eps)\right)\\
&\quad - \dfrac{qe}{\omega_c}\wedge R^\eps - \dfrac{qe}{\omega_c}\wedge F^\eps - q \int_{\R^3}\dfrac{(v\cdot e)e}{\eps}f^\eps\mathrm{d}v.
\end{align*}
Therefore, for the second term in \eqref{equ:time_deri_E}, using $\cur E =0$ and $E\cdot e =0$ in \eqref{equ:lim_mod_2dx} for the electric field $E$, we have
\begin{equation}
\label{equ:time_deri_E2}
\begin{split}
&- \dfrac{\epsilon_0}{m}\int_{\R^2} E\cdot (\partial_t E^\eps - \partial_t E)\mathrm{d}x_\perp \\
&=  \dfrac{q}{m}\int_{\R^2} E\cdot \left[\dfrac{e}{\omega_c} \wedge (\sigma \nabla_x n^\eps - \dfrac{\epsilon_0}{m}\mathcal{A}(E^\eps,E^\eps) - \dfrac{q}{m}DE^\eps) - \dfrac{n e}{\omega_c}\wedge k[n]\right]\mathrm{d}x_\perp \\
&\quad - \dfrac{q}{m}\int_{\R^2}E\cdot\left(\dfrac{e}{\omega_c}\wedge \dfrac{1}{m\mu_0}\mathcal{A}(B^\eps,B^\eps)\right)\mathrm{d}x_\perp + \dfrac{q}{m}\int_{\R^2}E\cdot\left(\dfrac{e}{\omega_c}\wedge (\dfrac{\epsilon_0\eps}{m}\partial_t(E^\eps\wedge B^\eps))\right)\mathrm{d}x_\perp\\
&\quad + \dfrac{q}{m}\int_{\R^2}E\cdot\left(\dfrac{e}{\omega_c}\wedge R^\eps\right)\mathrm{d}x_\perp  + \dfrac{q}{m}\int_{\R^2}E\cdot\left(\dfrac{e}{\omega_c}\wedge F^\eps\right)\mathrm{d}x_\perp.
\end{split}
\end{equation}
Combining \eqref{equ:time_deri_mag}-\eqref{equ:time_deri_E2}, we obtain the desired result.
\end{proof}

We need to compute the following terms which appears in Proposition \ref{TimeDeriv}
\begin{align*}
 \int_{\R^2} W[n]\cdot \left( \dfrac{\epsilon_0\eps}{m} \partial_t(E^\eps\wedge B^\eps)-\dfrac{\epsilon_0}{m}\mathcal{A}(E^\eps,E^\eps) -\dfrac{1}{m\mu_0}\mathcal{A}(B^\eps,B^\eps)\right)\mathrm{d}x_\perp.
\end{align*}
\begin{lemma}
\label{last_contribu}
We have 
\begin{align*}
&\int_{\R^2} W[n] \cdot \left( \dfrac{\epsilon_0\eps}{m} \partial_t(E^\eps\wedge B^\eps)-\dfrac{\epsilon_0}{m}\mathcal{A}(E^\eps,E^\eps)-\dfrac{1}{m\mu_0}\mathcal{A}(B^\eps,B^\eps)\right)\mathrm{d}x_\perp \\
&= \dfrac{\mathrm{d}}{\mathrm{d}t}\int_{\R^2} W[n] \cdot \left[ \epsilon_0\eps  ((E^\eps - E)\wedge (B^\eps -\eps B_1))\right]\mathrm{d}x_\perp \\
&\quad -\int_{\R^2} \partial_t W[n] \cdot \left[ \epsilon_0\eps  ((E^\eps - E)\wedge (B^\eps -\eps B_1))\right]\mathrm{d}x_\perp \\
&\quad  + \int_{\R^2} D_x W[n] : \epsilon_0 \left((E^\eps-E)\otimes(E^\eps-E) - \dfrac{1}{2}|E^\eps-E|^2 I_3\right)\mathrm{d}x_\perp\\
&\quad + \int_{\R^2} D_x W[n] : \dfrac{1}{\mu_0}\left((B^\eps-\eps B_1)\otimes(B^\eps - \eps B_1) - \dfrac{1}{2}|B^\eps-\eps B_1|^2 I_3\right) \mathrm{d}x_\perp\\
&\quad - \dfrac{q}{m}\int_{\R^2} W[n] \cdot ((n-D) E^\eps)\mathrm{d}x_\perp  - \dfrac{\epsilon_0}{m}\int_{\R^2}W[n]\cdot (E\Divx (E^\eps -E))\mathrm{d}x_\perp\\
 &\quad +\dfrac{\epsilon_0\eps}{m}\int_{\R^2}W[n]\cdot \left((E^\eps-E)\wedge \partial_t(\eps B_1)\right)\mathrm{d}x_\perp - \dfrac{1}{m}\int_{\R^2} (nW[n])\cdot (J^\eps\wedge(\eps B_1))\mathrm{d}x_\perp.
\end{align*}
\end{lemma}
\begin{proof}
The evolution in time of $\epsilon_0\eps \partial_t (E^\eps\wedge B^\eps)$ can be expressing by using the Maxwell equations \eqref{equ:MaxwellEps} as follows:
\begin{equation}
\label{equ:mixedEBeps}
\begin{split}
\epsilon_0\eps \partial_t (E^\eps\wedge B^\eps) &= \epsilon_0\eps \partial_t ((E^\eps - E)\wedge (B^\eps -\eps B_1)) + \epsilon_0\eps \partial_t (E\wedge \eps B_1)\\
&\quad + \epsilon_0\eps \partial_t(E\wedge(B^\eps -\eps B_1)) + \epsilon_0\eps \partial_t ((E^\eps-E)\wedge (\eps B_1))\\
 &= \epsilon_0\eps \partial_t ((E^\eps - E)\wedge (B^\eps -\eps B_1)) + \epsilon_0\eps \partial_t (E\wedge \eps B_1)\\
&\quad + \epsilon_0\eps \partial_t E \wedge (B^\eps -\eps B_1) + \epsilon_0 E\wedge (-\cur (E^\eps -E) - \eps \partial_t(\eps B_1))\\
&\quad + \left(\dfrac{1}{\mu_0}\cur (B^\eps - \eps B_1) - J^\eps + \eps q \dfrac{n e}{\omega_c}\wedge k[n]  \right)\wedge (\eps B_1)\\
&\quad + \epsilon_0\eps (E^\eps - E)\wedge \partial_t (\eps B_1).
\end{split}
\end{equation}
We will now compute $\epsilon_0\eps \partial_t(E\wedge \eps B_1)$. First, using the formula of $\mathcal{A}(u,u)$ in \eqref{equ:quad_field}, we have
\begin{equation}
\label{equ:EEeps}
\begin{split}
\epsilon_0 \mathcal{A}(E^\eps,E^\eps) &= \epsilon_0 \mathcal{A}(E^\eps-E,E^\eps-E) + \epsilon_0 \mathcal{A}(E,E)\\
&\quad + \epsilon_0 (E^\eps - E)\Divx E + \epsilon_0 E\Divx(E^\eps-E) - \epsilon_0 E\wedge \cur(E^\eps-E),
\end{split}
\end{equation}
and
\begin{equation}
\label{equ:BBeps}
\begin{split}
\dfrac{1}{\mu_0}\mathcal{A}(B^\eps,B^\eps) &= \dfrac{1}{\mu_0}\mathcal{A}(B^\eps-\eps B_1,B^\eps - \eps B_1)+ \dfrac{1}{\mu_0}\mathcal{A}(\eps B_1,\eps B_1)\\
&\quad -\dfrac{1}{\mu_0}(B^\eps -\eps B_1)\wedge \cur (\eps B_1) - \dfrac{1}{\mu_0}(\eps B_1)\wedge \cur(B^\eps - \eps B_1).
\end{split}
\end{equation}
We then compute the following contribution
\begin{equation}
\label{equ:mixedEB}
\begin{split}
&\epsilon_0\eps \partial_t (E\wedge \eps B_1) - \epsilon_0 \mathcal{A}(E,E) - \dfrac{1}{\mu_0}\mathcal{A}(\eps B_1,\eps B_1)\\&= \eps \left( \dfrac{1}{\mu_0}\cur B_1 - q \dfrac{n e}{\omega_c}\wedge k[n] \right)\wedge (\eps B_1) + \epsilon_0\eps E\wedge \partial_t(\eps B_1)\\
&\quad - \epsilon_0 E \Divx E + \dfrac{1}{\mu_0}(\eps B_1)\wedge \cur (\eps B_1)\\
&= -\eps q \left( \dfrac{n e}{\omega_c}\wedge k[n]\right) \wedge (\eps B_1) + \epsilon_0\eps E\wedge \partial_t(\eps B_1)
 - \epsilon_0 E \Divx E.
 \end{split}
\end{equation}
Therefore, from \eqref{equ:EEeps}, \eqref{equ:BBeps} and \eqref{equ:mixedEB}, we deduce that
\begin{equation}
\label{equ:mixedEB_bis}
\begin{split}
&\epsilon_0\eps \partial_t (E\wedge \eps B_1) \\
&= \epsilon_0 \mathcal{A}(E^\eps,E^\eps) + \dfrac{1}{\mu_0}\mathcal{A}(B^\eps,B^\eps)- \epsilon_0 \mathcal{A}(E^\eps-E,E^\eps-E) - \dfrac{1}{\mu_0}\mathcal{A}(B^\eps-\eps B_1,B^\eps - \eps B_1)\\
&\quad -  \epsilon_0 (E^\eps - E)\Divx E - \epsilon_0 E\Divx(E^\eps-E) + \epsilon_0 E\wedge \cur(E^\eps-E)\\
&\quad + \dfrac{1}{\mu_0}(B^\eps -\eps B_1)\wedge \cur (\eps B_1) + \dfrac{1}{\mu_0}(\eps B_1)\wedge \cur(B^\eps - \eps B_1)\\
&\quad -\eps q \left( \dfrac{n e}{\omega_c}\wedge k[n]\right) \wedge (\eps B_1) + \epsilon_0\eps E\wedge \partial_t(\eps B_1)
 - \epsilon_0 E \Divx E.
\end{split}
\end{equation}
Combining \eqref{equ:mixedEBeps} and \eqref{equ:mixedEB_bis}, we obtain
\begin{equation*}
\begin{split}
& \epsilon_0\eps \partial_t (E^\eps\wedge B^\eps) - \epsilon_0 \mathcal{A}(E^\eps,E^\eps) - \dfrac{1}{\mu_0}\mathcal{A}(B^\eps,B^\eps)\\
&=\epsilon_0\eps \partial_t ((E^\eps - E)\wedge (B^\eps -\eps B_1)) - \epsilon_0 \mathcal{A}(E^\eps-E,E^\eps-E) - \dfrac{1}{\mu_0}\mathcal{A}(B^\eps-\eps B_1,B^\eps - \eps B_1)\\
&\quad - q(n-D)E^\eps - \epsilon_0 E \Divx (E^\eps -E) - \left( q\dfrac{n e}{\omega_c}\wedge k[n] \right)\wedge (B^\eps - \eps B_1)\\
&\quad +\epsilon_0\eps (E^\eps-E)\wedge \partial_t(\eps B_1) - J^\eps\wedge(\eps B_1),
\end{split}
\end{equation*}
which yields the desired result.
\end{proof}

Now we estimate the following contributions in Proposition \ref{TimeDeriv} and Lemma \ref{last_contribu}
\begin{align*}
&\dfrac{q}{m}\int_{\R^2}nW[n]\cdot (E^\eps -E)\mathrm{d}x_\perp + \sigma \int_{\R^2}\Divx (n W[n])\dfrac{n^\eps -n}{n}\mathrm{d}x_\perp \\
&\quad + \int_{\R^2}W[n] \cdot  (\sigma \nabla_x n^\eps -\dfrac{q}{m}DE^\eps)\mathrm{d}x_\perp
- \dfrac{q}{m}\int_{\R^2}W[n] \cdot ((n-D) E^\eps)\mathrm{d}x_\perp\\  
&\quad - \dfrac{\epsilon_0}{m}\int_{\R^2} W[n]\cdot (E\Divx (E^\eps -E))\mathrm{d}x_\perp\\
&= \dfrac{q}{m}\int_{\R^2}nW[n]\cdot (E^\eps -E)\mathrm{d}x_\perp +\sigma \int_{\R^2}\Divx (W[n])(n^\eps -n)\mathrm{d}x_\perp \\
&\quad + \int_{\R^2} W[n]\cdot\sigma \dfrac{\nabla_x n}{n}(n^\eps -n )\mathrm{d}x_\perp
- \sigma \int_{\R^2}\Divx(W[n]) n^\eps\mathrm{d}x_\perp\\
&\quad -  \dfrac{q}{m}\int_{\R^2}n W[n]\cdot  E^\eps\mathrm{d}x_\perp  - \dfrac{q}{m}\int_{\R^2}W[n]\cdot  (n^\eps -n)E\mathrm{d}x_\perp\\
& = - \dfrac{q}{m}\int_{\R^2}n W[n]\cdot  E\mathrm{d}x_\perp - \sigma \int_{\R^2}\Divx (W[n])n\mathrm{d}x_\perp\\
&\quad +\dfrac{q}{m}\int_{\R^2}W[n]\cdot E (n^\eps -n)\mathrm{d}x_\perp - \dfrac{q}{m}\int_{\R^2}W[n]\cdot  (n^\eps -n)E\mathrm{d}x_\perp\\
&= 0.
\end{align*}

Coming back to \eqref{BalModEnerDens}, the modulated energy balance becomes
\begin{equation}
\label{BalModEnerDensBis}
\begin{split}
&\calE[n^\eps(t)|n(t)] +  \sigma\int_{\R^2}\int_{\R^3}{n^\eps(t) M h\left(\dfrac{f^\eps(t)}{n^\eps(t) M} \right)}\mathrm{d}v\mathrm{d}x_\perp + \dfrac{1}{2}\int_{\R^2}\int_{\Sp^2}\mathrm{d}m^\eps\\
&\quad + \dfrac{1}{2}\int_{\R^2} \mathrm{d}\left( \dfrac{\epsilon_0}{m}tr(\mu^\eps_E) + \dfrac{\mu_0}{m}tr(\mu^\eps_B)\right)  + \dfrac{1}{\eps\tau}\int_0^t\int_{\R^2}\int_{\R^3}{\dfrac{|\sigma \nabla_v f^\eps + v f^\eps|^2}{f^\eps}}\mathrm{d}v\mathrm{d}x_\perp\mathrm{d}s\\
&\leq \calE[n^\eps(0)|n(0)] +  \sigma\int_{\R^2}\int_{\R^3}{n^\eps(0) M h\left(\dfrac{f^\eps(0)}{n^\eps(0) M} \right)}\mathrm{d}v\mathrm{d}x_\perp \\
&\quad + \dfrac{\eps^2}{2\mu_0 m}\left(\int_{\R^2}|B_1(t)|^2\mathrm{d}x_\perp - \int_{\R^2}|B_1(0)|^2\mathrm{d}x_\perp \right) - \dfrac{\eps}{\mu_0 m}\int_0^t\int_{\R^2} B^\eps\cdot \partial_s B_1\mathrm{d}x_\perp\mathrm{d}s\\
&\quad + \int_0^t\dfrac{\mathrm{d}}{\mathrm{d}s}\int_{\R^2} W[n] \cdot \left[ \epsilon_0\eps  ((E^\eps - E)\wedge (B^\eps -\eps B_1))\right]\mathrm{d}x_\perp\mathrm{d}s \\
&\quad -\int_0^t\int_{\R^2} \partial_s W[n] \cdot \left[ \epsilon_0\eps  ((E^\eps - E)\wedge (B^\eps -\eps B_1))\right]\mathrm{d}x_\perp\mathrm{d}s \\
&\quad  + \int_0^t\int_{\R^2} D_x W[n] : \epsilon_0 \left((E^\eps-E)\otimes(E^\eps-E) - \dfrac{1}{2}|E^\eps-E|^2 I_3\right)\mathrm{d}x_\perp\mathrm{d}s\\
&\quad + \int_0^t\int_{\R^2} D_x W[n] : \dfrac{1}{\mu_0}\left((B^\eps-\eps B_1)\otimes(B^\eps - \eps B_1) - \dfrac{1}{2}|B^\eps-\eps B_1|^2 I_3\right) \mathrm{d}x_\perp\mathrm{d}s\\
 &\quad +\dfrac{\epsilon_0\eps}{m}\int_0^t\int_{\R^2}W[n]\cdot \left((E^\eps-E)\wedge \partial_s(\eps B_1)\right)\mathrm{d}x_\perp\mathrm{d}s \\
 &\quad - \dfrac{1}{m}\int_0^t\int_{\R^2} (nW[n])\cdot (J^\eps\wedge(\eps B_1))\mathrm{d}x_\perp\mathrm{d}s\\
&\quad + \int_0^t\int_{\R^2}W[n]\cdot R^\eps\mathrm{d}x_\perp\mathrm{d}s
+ \int_0^t\int_{\R^2}W[n]\cdot F^\eps\mathrm{d}x_\perp\mathrm{d}s,
\end{split}
\end{equation}
where
\begin{equation}
\label{equ:Feps}
\begin{split}
&\int_0^t \int_{\R^2} W[n]\cdot F^\eps\mathrm{d}x_\perp\mathrm{d}s \\
&= - \int_0^t\int_{\R^2} D_x W[n] : \left(\int_{\R^3}(\sigma\nabla_v f^\eps + v f^\eps)\otimes v \mathrm{d}v\right) \mathrm{d}x_\perp\mathrm{d}s \\
&\quad + \eps \int_0^t \dfrac{\mathrm{d}}{\mathrm{d}s}\int_{\R^2} W[n]\cdot \left(\int_{\R^3}v f^\eps\mathrm{d}v\right) \mathrm{d}x_\perp\mathrm{d}s - \eps\int_0^t \int_{\R^2}\partial_s W[n]\cdot \left(\int_{\R^3}v f^\eps\mathrm{d}v \right)\mathrm{d}x_\perp\mathrm{d}s \\
&\quad + \dfrac{1}{\tau} \int_0^t\int_{\R^2} W[n]\cdot \left(\int_{\R^3} vf^\eps\mathrm{d}v\right)\mathrm{d}x_\perp\mathrm{d}s,
\end{split}
\end{equation}
and the integral of the group of defect measures $R^\eps$ can be rewritten as
\begin{equation}
\label{equ:esti_defect}
\begin{split}
&\int_0^t \int_{\R^2} W[n]\cdot R^\eps\mathrm{d}x_\perp\mathrm{d}s
\\
&= \int_0^t\int_{\R^2} D_x W[n] : \left[ \dfrac{\epsilon_0}{m}\left( \mu^\eps_E - \dfrac{1}{2}tr(\mu^\eps_E)I_3 \right) + \dfrac{1}{m\mu_0}\left(\mu^\eps_B - \dfrac{1}{2}tr(\mu^\eps_B)I_3 \right) \right]\\
&\quad + \dfrac{\epsilon_0\eps}{m}\int_0^t\dfrac{\mathrm{d}}{\mathrm{d}s}\int_{\R^2}W[n]\cdot \mathrm{d}(\mu^\eps_{EB}) - \dfrac{\epsilon_0\eps}{m}\int_0^t\int_{\R^2}\partial_s W[n]\cdot \mathrm{d}(\mu^\eps_{EB})\\
&\quad - \int_0^t\int_{\R^2} D_x W[n] : \int_{\Sp^2}\sigma\otimes\sigma \mathrm{d}m^\eps.
\end{split}
\end{equation}
 In order to apply the Gronwall lemma, we estimate the terms in the right hand side of \eqref{BalModEnerDensBis}.
\begin{coro}
\label{Esti_modu_Gron}
The modulated energy $\mathcal{E}[n^\eps(t)|n(t)]$ satisfies the Gronwall inequality
\begin{equation}
\label{Gronwall_mod}
\begin{split}
&\calE[n^\eps(t)|n(t)] +  \sigma\int_{\R^2}\int_{\R^3}{n^\eps(t) M h\left(\dfrac{f^\eps(t)}{n^\eps(t) M} \right)}\mathrm{d}v\mathrm{d}x_\perp + \dfrac{1}{2\eps\tau}\int_0^t\int_{\R^2}\int_{\R^3}{\dfrac{|\sigma \nabla_v f^\eps + v f^\eps|^2}{f^\eps}}\mathrm{d}v\mathrm{d}x_\perp\mathrm{d}s\\
&\quad +  \dfrac{1}{2}\int_{\R^2}\int_{\Sp^2}\mathrm{d}m^\eps+
\dfrac{1}{2}\int_{\R^2} \mathrm{d}\left( \dfrac{\epsilon_0}{m}tr(\mu^\eps_E) + \dfrac{\mu_0}{m}tr(\mu^\eps_B)\right)  \\
&\leq (1+ C_T \eps)\calE[n^\eps(0)|n(0)] +  \sigma\int_{\R^2}\int_{\R^3}{n^\eps(0) M h\left(\dfrac{f^\eps(0)}{n^\eps(0) M} \right)}\mathrm{d}v\mathrm{d}x_\perp  \\ &\quad + C_T \int_0^t \mathcal{E}[n^\eps(s)|n(s)]\mathrm{d}s 
 + C_T \eps \mathcal{E}[n^\eps(t)|n(t)] + C_T \eps \int_{\R^2}\dfrac{1}{2}\mathrm{d}\left(\dfrac{\epsilon_0}{m}tr(\mu^\eps_E) + \dfrac{\mu_0}{m} tr(\mu^\eps_B) \right) \\
&\quad + C_T ( \int_0^t\int_{\R^2}\int_{\Sp^2}\dfrac{1}{2}\mathrm{d}m^\eps + \int_0^t\int_{\R^2}\dfrac{1}{2}\mathrm{d}\left(\dfrac{\epsilon_0}{m}tr(\mu^\eps_E) + \dfrac{\mu_0}{m} tr(\mu^\eps_B) \right)) + C_T \sqrt{\eps}.
\end{split}
\end{equation}
\end{coro}
\begin{proof}
Using the Cauchy-Schwarz inequality and the properties of solution of \eqref{equ:lim_mod_2dx}, we have 
\begin{align*}
&- \dfrac{\eps}{\mu_0 m}\int_0^t\int_{\R^2} B^\eps\cdot \partial_s B_1\mathrm{d}x_\perp\mathrm{d}s + \dfrac{\epsilon_0\eps}{m}\int_0^t\int_{\R^2}W[n]\cdot (E^\eps-E)\wedge \partial_s(\eps B_1)\mathrm{d}x_\perp\mathrm{d}s \\
&\leq \dfrac{1}{2\mu_0 m} \int_0^t \int_{\R^2}|B^\eps -\eps B_1|^2\mathrm{d}x_\perp\mathrm{d}s + \dfrac{\epsilon_0}{2m}\int_0^t\int_{\R^2} |E^\eps-E|^2\mathrm{d}x_\perp\mathrm{d}s\\
&\quad + \dfrac{\|k[n]\|_{L^\infty}}{B_0} \eps^2\int_0^t\int_{\R^2}(|B_1|^2+ |\partial_s B_1|^2)\mathrm{d}x_\perp\mathrm{d}s\\
&\leq \int_0^t\mathcal{E}[n^\eps(s)|n(s)]\mathrm{d}s + \mathcal{O}(\eps^2).
\end{align*}
 Similarly, by standard estimates, we also have
\begin{align*}
& \int_0^t\dfrac{\mathrm{d}}{\mathrm{d}s}\int_{\R^2} W[n] \cdot \left[ \epsilon_0\eps  ((E^\eps - E)\wedge (B^\eps -\eps B_1))\right]\mathrm{d}x_\perp\mathrm{d}s \\
&\quad -\int_0^t\int_{\R^2} \partial_s W[n] \cdot \left[ \epsilon_0\eps  ((E^\eps - E)\wedge (B^\eps -\eps B_1))\right]\mathrm{d}x_\perp\mathrm{d}s \\
&\quad  + \int_0^t\int_{\R^2} D_x W[n] : \epsilon_0 \left((E^\eps-E)\otimes(E^\eps-E) - \dfrac{1}{2}|E^\eps-E|^2 I_3\right)\mathrm{d}x_\perp\mathrm{d}s\\
&\quad + \int_0^t\int_{\R^2} D_x W[n] : \dfrac{1}{\mu_0}\left((B^\eps-\eps B_1)\otimes(B^\eps - \eps B_1) - \dfrac{1}{2}|B^\eps-\eps B_1|^2 I_3\right) \mathrm{d}x_\perp\mathrm{d}s\\
&\leq C_T \eps \mathcal{E}[n^\eps(t)|n(t)] + C_T \eps \mathcal{E}[n^\eps(0)|n(0)] + C_T \int_0^t \mathcal{E}[n^\eps(s)|n(s)]\mathrm{d}s,
\end{align*}
where $C_T$ is a positive constant depending on $\|k[n]\|_{W^{1,\infty}}$ and $B_0$.

We estimate now the defect terms appearing in \eqref{equ:esti_defect}. As $\mu^\eps_E$ and $\mu^\eps_B$ are nonnegative symmetric matrix, taking into account that for any matrix $A\in C^0(\R^2)^9$ we have the inequalities
\[
\left|\int_{\R^2} A(x_\perp) : \mathrm{d}(\mu^\eps_E)\right| \leq \int_{\R^2} |A(x_\perp)|\mathrm{d} tr(\mu^\eps_E),\quad\quad \left|\int_{\R^2} A(x_\perp) : \mathrm{d}(\mu^\eps_B)\right| \leq \int_{\R^2} |A(x_\perp)|\mathrm{d}tr(\mu^\eps_B),
\]
we deduce that
\begin{align*}
\int_0^t\int_{\R^2} D_x W[n] : \left[ \dfrac{\epsilon_0}{m}\left( \mu^\eps_E - \dfrac{1}{2}tr(\mu^\eps_E)I_3 \right) + \dfrac{1}{m\mu_0}\left(\mu^\eps_B - \dfrac{1}{2}tr(\mu^\eps_B)I_3 \right) \right]\\
\leq C_T \int_0^t \int_{\R^2} \mathrm{d}\left(\dfrac{\epsilon_0}{m}tr(\mu^\eps_E) + \dfrac{\mu_0}{m} tr(\mu^\eps_B) \right). 
\end{align*}
In the same way for any vector $a\in C^0(\R^2)^3$ we have
\[
\left| \int_{\R^2} a(x_\perp)\cdot \mathrm{d}\mu^\eps_{EB} \right| \leq \dfrac{1}{2}\int_{\R^2} |a(x_\perp)|\mathrm{d} tr(\mu^\eps_E) + \dfrac{1}{2}\int_{\R^2}|a(x_\perp)|\mathrm{d} tr(\mu^\eps_B),
\]
implying that 
\begin{align*}
&\dfrac{\epsilon_0\eps}{m}\int_0^t\dfrac{\mathrm{d}}{\mathrm{d}s}\int_{\R^2} W[n]\cdot \mathrm{d}\mu^\eps_{EB} - \dfrac{\epsilon_0\eps}{m}\int_0^t\int_{\R^2}\partial_s W[n] \cdot \mathrm{d}\mu^\eps_{EB}\\
&\leq C_T\eps \int_{\R^2}\dfrac{1}{2}\mathrm{d}\left(\dfrac{\epsilon_0}{m}tr(\mu^\eps_E) + \dfrac{\mu_0}{m} tr(\mu^\eps_B) \right)+ C_T  \int_0^t\int_{\R^2}\dfrac{1}{2}\mathrm{d}\left(\dfrac{\epsilon_0}{m}tr(\mu^\eps_E) + \dfrac{\mu_0}{m} tr(\mu^\eps_B) \right).
\end{align*}
Similarly, we have
\[
\int_0^t\int_{\R^2}D_x W[n] : \int_{\Sp^2}\sigma\otimes\sigma \mathrm{d}m^\eps \leq C_T \int_0^t\int_{\R^2}\int_{\Sp^2}\dfrac{1}{2}\mathrm{d}m^\eps.
\]
Therefore, we obtain
\begin{align*}
\int_0^t\int_{\R^2} W[n]\cdot R^\eps\mathrm{d}x_\perp\mathrm{d}s &\leq C_T  \eps\int_{\R^2}\dfrac{1}{2}\mathrm{d}\left(\dfrac{\epsilon_0}{m}tr(\mu^\eps_E)+ \dfrac{\mu_0}{m} tr(\mu^\eps_B) \right)\\
&\quad +  C_T \left(\dfrac{1}{2}\int_0^t\int_{\R^2}\int_{\Sp^2}\mathrm{d}m^\eps +  \dfrac{1}{2}\int_0^t\int_{\R^2}\mathrm{d}\left(\dfrac{\epsilon_0}{m}tr(\mu^\eps_E) + \dfrac{\mu_0}{m} tr(\mu^\eps_B) \right) \right).
\end{align*}

We estimate now the last integral $\int_0^t\int_{\R^2}W[n]\cdot F^\eps\mathrm{d}x_\perp\mathrm{d}s$ in \eqref{BalModEnerDensBis} and $-\frac{1}{m}\int_0^t\int_{\R^2}(nW[n])\cdot(J^\eps\wedge(\eps B_1))\mathrm{d}x_\perp\mathrm{d}s$. We have for some value $C$ to be precised later on
\begin{align*}
& - \int_0^t\int_{\R^2} D_x W[n] : \int_{\R^3}{(\sigma \nabla_v f^\eps + f^\eps v)\otimes v}\mathrm{d}v\mathrm{d}x_\perp\mathrm{d}s\\
 &\leq C_T\left[\dfrac{1}{2\eps\tau C}\int_0^t\int_{\R^2}\int_{\R^3}{\dfrac{|\sigma \nabla_v f^\eps + f^\eps v |^2}{f^\eps}}\mathrm{d}v\mathrm{d}x_\perp\mathrm{d}s + \eps \tau C \int_0^t\int_{\R^2}\int_{\R^3}{f^\eps \dfrac{|v|^2}{2}}\mathrm{d}v\mathrm{d}x_\perp\mathrm{d}s \right].
\end{align*}
Since $\int_{\R^3}vf^\eps\mathrm{d}v = \int_{\R^3}{(\sigma \nabla_v f^\eps + f^\eps v )}\mathrm{d}v$, we have
\begin{align*}
&\int_{0}^{t}\int_{\R^2}{W[n(s)]\cdot (\eps \partial_s \int_{\R^3}v f^\eps\mathrm{d}v + \dfrac{1}{\tau}\int_{\R^3}v f^\eps\mathrm{d}v)}\mathrm{d}x_\perp\mathrm{d}s \\
&=  \eps \int_{\R^2}{W[n(t)]\cdot \int_{\R^3}vf^\eps(t)\mathrm{d}v}\mathrm{d}x_\perp - \eps \int_{\R^2}{W[n(0)]\cdot \int_{\R^3}vf^\eps(0,x)\mathrm{d}v}\mathrm{d}x_\perp\\
&\quad + \int_0^t\int_{\R^2}\int_{\R^3}{[\sigma \nabla_v f^\eps + f^\eps(s,x,v)v]\cdot \left[ \dfrac{W[n(s)]}{\tau} - \eps \partial_s W[n(s)] \right]}\mathrm{d}v\mathrm{d}x_\perp\mathrm{d}s\\
&\leq \sqrt{\eps} \int_{\R^2}\int_{\R^3}{(f^\eps(0,x,v) + f^\eps(t,x,v))\left(\eps\dfrac{|v|^2}{2}+ \dfrac{\| W[n]\|^2 _{L^\infty}}{2} \right)}\mathrm{d}v\mathrm{d}x_\perp\\
&\quad+ \left[\eps \| \partial_s W[n] \|_{L^\infty} + \dfrac{\|W[n]\|_{L^\infty}}{\tau}  \right]\int_0^t\int_{\R^2}\int_{\R^3}{\left\{ \dfrac{1}{2\eps C}\dfrac{|\sigma \nabla_v f^\eps + f^\eps v |^2}{f^\eps}+ \dfrac{\eps C}{2}f^\eps \right\}}\mathrm{d}v\mathrm{d}x_\perp\mathrm{d}s.
\end{align*}
Similarly, we also have
\begin{align*}
&-\frac{1}{m}\int_0^t\int_{\R^2}(nW[n])\cdot(J^\eps\wedge(\eps B_1))\mathrm{d}x_\perp\mathrm{d}s \\
&\leq \eps \| n\|_{L^\infty}\|W[n]\|_{L^\infty} \|B_1\|_{L^\infty} \int_0^t\int_{\R^2}\int_{\R^3}{\left\{ \dfrac{1}{2\eps C}\dfrac{|\sigma \nabla_v f^\eps + f^\eps v |^2}{f^\eps}+ \dfrac{\eps C}{2}f^\eps \right\}}\mathrm{d}v\mathrm{d}x_\perp\mathrm{d}s.
\end{align*}
Taking $C$ large enough, we obtain by Lemma \ref{KinEne} that
\begin{align*}
-\frac{1}{m}\int_0^t\int_{\R^2}(nW[n])\cdot(J^\eps\wedge(\eps B_1))\mathrm{d}x_\perp\mathrm{d}s + \int_0^t \int_{\R^2} W[n]\cdot F^\eps\mathrm{d}x_\perp \mathrm{d}s\\
 \leq \dfrac{1}{2\eps\tau}\int_0^t\int_{\R^2}\int_{\R^3} \dfrac{|\sigma\nabla_v f^\eps + v f^\eps|^2}{f^\eps}\mathrm{d}v\mathrm{d}x_{\perp}\mathrm{d}s + C_T\sqrt{\eps}. 
\end{align*}
Putting together all the above estimates, we finally obtain the  inequality \eqref{Gronwall_mod}.
\end{proof}

We are now ready to prove our main theorem.
\begin{proof}(of Theorem \ref{MainThm}).
Applying Gronwall's lemma to inequality \eqref{Gronwall_mod}, we deduce that, for some constant $C_T>0$, and for all $0\leq t\leq T$, $0< \eps \leq 1$, the following holds:
\begin{align*}
\calE[n^\eps(t)|n(t)] +&  \sigma\int_{\R^2}\int_{\R^3}{n^\eps(t) M h\left(\dfrac{f^\eps(t)}{n^\eps(t) M} \right)}\mathrm{d}v\mathrm{d}x_\perp + \dfrac{1}{2\eps\tau}\int_0^t \int_{\R^2}\int_{\R^3}{\dfrac{|\sigma \nabla_v f^\eps + v f^\eps|^2}{f^\eps}}\mathrm{d}v\mathrm{d}x_\perp\mathrm{d}s\\
&+ \dfrac{1}{2}\int_{\R^2}\int_{\Sp^2}\mathrm{d}m^\eps+
\dfrac{1}{2}\int_{\R^2} \mathrm{d}\left( \dfrac{\epsilon_0}{m}tr(\mu^\eps_E) + \dfrac{\mu_0}{m}tr(\mu^\eps_B)\right)\\
&\leq  \left[\calE[n^\eps(0)|n(0)] +  \sigma\int_{\R^2}\int_{\R^3}{n^\eps(0) M h\left(\dfrac{f^\eps(0)}{n^\eps(0) M} \right)}\mathrm{d}v\mathrm{d}x_\perp  + C_T \sqrt{\eps}\right]e^{C_T t}.
\end{align*}
The above inequality says that the particle density $f^\eps$ remains close to the Maxwellian with the same concentration, i.e., $n^\eps(t)M$, and $n^\eps(t)$ stays near $n(t)$, provided that analogous behaviour occur for the initial conditions.

We justify the convergence of $f^\eps$ toward $nM$ in $L^\infty(0,T;L^1(\R^2\times\R^3))$, the other convergences being obvious. We use the Csis\'ar -Kullback inequality in order to control the $L^1$ norm by the relative entropy, cf. \cite{Csi1967, Kul1967}
\[
\int_{\R^n}|g-g_0|\mathrm{d} x \leq 2 \max \left\{ \left( \int_{\R^n}g_0\mathrm{d} x \right)^{1/2}, \left( \int_{\R^n}g\mathrm{d} x\right)^{1/2} \right\}\left(\int_{\R^n}g_0  h\left(\dfrac{g}{g_0} \right)\mathrm{d} x \right)^{1/2},
\]
for any nonnegative integrable functions $g_0,g: \R^n \to \R$. Applying two times the  Csis\'ar -Kullback inequality, we obtain
\begin{align*}
&\int_{\R^2}\int_{\R^3}{|f^\eps(t,x_\perp,v) -n(t,x_\perp)M(v)|}\mathrm{d}v\mathrm{d}x_\perp \\ &\leq \int_{\R^2}\int_{\R^3}{|f^\eps(t,x_\perp,v)- n^\eps(t,x_\perp)M(v)|}\mathrm{d}v\mathrm{d}x_\perp + \int_{\R^2}{|n^\eps(t,x_\perp)-n(t,x_\perp)|}\mathrm{d}x_\perp\\
&\leq 2 \sqrt{M_0} \left(\int_{\R^2}\int_{\R^3}{ n^\eps(t)M(v) h\left(\dfrac{f^\eps(t)}{n^\eps(t)M} \right)}\mathrm{d}v\mathrm{d}x_\perp\right)^{1/2} \\&\quad + 2 \max\left\{ \sqrt{M_0},\sqrt{\|n_{0}\|_{L^1(\R^2)}} \right\} \left( \int_{\R^2}{n(t) h \left(\dfrac{n^\eps(t)}{n(t)}\right)}\mathrm{d}x_\perp  \right)^{1/2}  \to 0 ,\,\,\,\,\mathrm{as}\,\eps\searrow 0.
\end{align*}
\end{proof}

\section{Appendix: Global existence solution}
In this section, we prove the existence of global in time weak solution for the VMFP equations \eqref{equ:VMFP-Scale}-\eqref{equ:Initial} stated in Theorem \ref{exis_weak_sol}. In order to simplify the
proofs of existence of the solution, as we do not want any uniform
estimate with respect to $\eps$, we will take $\eps=1$ and omit all the subscripts. We assume that the physical parameter $\tau,q,m,\mu_0,\epsilon_0$ and $\sigma$ are normalized. Thus, we consider the following problem:
\begin{equation}
\label{equ:VMFP_noeps}
\partial_t f + v\cdot\nabla_x f + (E + v\wedge(B + \textbf{B}_{\text{ext}}))\cdot\nabla_v f = \Divv(\nabla_v f + vf),
\end{equation}
\begin{equation}
\label{equ:FaAM_noeps}
\partial_t E - \cur B = - J = - \int_{\R^3}vf \mathrm{d}v,\quad \partial_t B + \cur E =0,
\end{equation}
\begin{equation}
\label{equ:Gauss_noeps}
\Divx E = \rho = \int_{\R^3}f \mathrm{d}v,\quad \Divx B =0.
\end{equation}
Notice that, once the two equations in \eqref{equ:FaAM_noeps} are satisfied, the two equations in \eqref{equ:Gauss_noeps} follow from the local conserversation law of mass \eqref{equ:law_charge_meas}.\\
We fix initial data
\begin{equation}
\label{equ:init_data_noeps}
\begin{split}
f_0\geq 0,\quad f_0 \in L^1\cap L^2(\R^3\times\R^3),\\
\int_{\R^3}\int_{\R^3}(|v|^2+ |\ln f_0|)f_0 \mathrm{d}x\mathrm{d}v <+\infty,\quad E_0, B_0\in L^2(\R^3),\\
\Divx B_0 = 0\,\,\,\,\text{in}\,\,\,\,\mathcal{D}',\quad \Divx E_0  = \rho_0\,\,\,\,\text{in}\,\,\,\,\mathcal{D}'.
\end{split}
\end{equation}

\textbf{The regularized VMFP system}. We first regularize the initial conditions for the VMFP system, i.e., we consider families $f^\delta_0\in \mathcal{D}_+(\R^3\times\R^3)$ and $E^\delta_0,B^\delta_0\in\mathcal{D}(\R^3;\R^3)$ so that
\[
\int_{\R^3}\int_{\R^3}|f_0 - f_0^\delta|(1+|v|^2) + |f_0 -f_0^\delta|^2 \mathrm{d}v\mathrm{d}x \to 0,\,\,\delta\to 0,
\]
and
\[
\int_{\R^3} |E_0 - E^\delta_0|^2 + |B_0 - B^\delta_0|^2\mathrm{d}x\to 0,\,\,\delta\to 0.
\]
We then consider the following regularized system of equations, in which the current density $J$ and the Lorentz force are regularized using a suitable mollifier $K_\delta$, as introduced in Subsection \ref{Prio_est}, i.e.,
\begin{equation}
\label{equ:regu_Vlasov}
\partial_t f^\delta + v\cdot\nabla_x f^\delta + [K_\delta(E^\delta +v\wedge B^\delta) + v\wedge \textbf{B}_{\text{ext}}]\cdot\nabla_v f^\delta = \Divv(\nabla_v f^\delta + vf^\delta),
\end{equation}
\begin{equation}
\label{equ:regu_FaAm}
\partial_t E^\delta - \cur B^\delta = - K_\delta J^\delta,\quad \partial_t B^\delta + \cur E^\delta =0,
\end{equation}

\textbf{The global well-posedness of the regularized system \eqref{equ:regu_Vlasov}-\eqref{equ:regu_FaAm}.}
\begin{pro}
\label{regu_solu}
Let $T>0$ and $\delta>0$ be fixed. Let $\textbf{B}_{\text{ext}}\in L^\infty(\R^3)$. Let
$f^\delta_0\in \mathcal{D}_+(\R^3\times\R^3)$ and  $E^\delta_0,\,\,B^\delta_0\in \mathcal{D}(\R^3;\R^3)$ satisfy \eqref{equ:init_data_noeps}.
Then there exists a unique weak solution  $0\leq f^\delta\in L^\infty(0,T;L^1\cap L^\infty(\R^3\times\R^3))$ and $E^\delta, B^\delta \in C([0,T];H^{s}(\R^3))\cap C^1([0,T],H^{s-1}(\R^3)),\,\,s \geq 5$, satisfying:
\begin{align*}
\| f^\delta\|_{L^\infty(0,T;L^1(\R^3\times\R^3))} = \|f^\delta_0\|_{L^1(\R^3\times\R^3)},\quad \|f^\delta\|_{L^\infty(0,T;L^\infty(\R^3\times\R^3))} \leq e^{3T}\|f^\delta_0\|_{L^\infty(\R^3\times\R^3)},\\
\|f^\delta\|^2_{L^\infty(0,T;L^2(\R^3\times\R^3))} + \| \nabla_v f^\delta\|^2_{L^2(\R^3\times\R^3)} \leq e^{2T} \|f^\delta_0\|^2_{L^2(\R^3\times\R^3)},
\end{align*}
and the free energy estimate \eqref{energy_sm}. Moreover, the solution verifies the local conservation laws of mass \eqref{equ:law_mass} and momentum \eqref{equ:law_moment} in the sense of distributions.
\end{pro}

The proof of Proposition \ref{regu_solu} will be devided into four steps. The first step is devoted to construct an iterative sequence for the regularized system. The second step establishes the local-time solution and the global-time solution is showed in the third step. Finally, in the fourth step, we provided the properties of the solution. 

\textbf{Step $1$: Construction of an iterative sequence.} We construct the solution of \eqref{equ:regu_Vlasov}-\eqref{equ:regu_FaAm} using the following iteration scheme. Set $f^\delta_0(t,x,v) = f^\delta_0(x,v)$, $(E^\delta_0(t,x),B^\delta_0(t,x))=(0,0)$. Then we define $(f^\delta_{l+1},E^\delta_{l+1},B^\delta_{l+1})$, for $l=0,1,...$, inductively as the solutions of the following linear equations:
\begin{equation}
\label{equ:lin_regu_Vlasov}
\partial_t f^\delta_{l+1} + v\cdot\nabla_x f^\delta_{l+1} + [K_\delta(E^\delta_l + v\wedge B^\delta_l)+ v\wedge \textbf{B}_{\text{ext}}]\cdot\nabla_v f^\delta_{l+1} = \Divv(\nabla_v f^\delta_{l+1}+vf^\delta_{l+1}),
\end{equation}
\begin{equation}
\label{equ:lin_regu_FaAm}
\partial_t E^\delta_{l+1} - \cur B^\delta_{l+1} = -K_\delta J^\delta_{l},\quad\quad \partial_tB^\delta_{l+1} +\cur E^\delta_{l+1} =0,
\end{equation}
\begin{equation}
\label{equ:lin_initial}
f^{\delta}_{l+1}(0,x,v) = f^\delta_0(x,v),\quad E^\delta_{l+1}(0,x ) = E^\delta_0(x),\quad B^\delta_{l+1}(0,x) = B^\delta_0(x).
\end{equation}

First, we solve Maxwell's equations \eqref{equ:lin_regu_FaAm} with a given right hand side, which form a linear symmetric hyperbolic system. Note that the Maxwell equations \eqref{equ:MaxwellEps} also have the structure of a symmetric hyperbolic  systems after multiplication by a symmetrizing matrix. Thanks to Theorem I in \cite{Kato}, the system admits a unique solution $(E^\delta_{l+1},B^\delta_{l+1})$ in $C([0,T],H^s(\R^3))\cap C^1([0,T],H^{s-1}(\R^3))$, for $s\geq 5$, provided that $K_\delta J^\delta_l\in L^1(0,T;H^s(\R^3))\cap C([0,T];H^{s-1}(\R^3))$, and the following  estimates hold:
\begin{equation}
\label{equ:esti_elecmag}
\sup_{t\in[0,T]}(\|E^\delta_{l+1}\|_{H^s} + \|B^\delta_{l+1}\|_{H^s}) \leq e^{CT}(\|E^\delta_0\|_{H^s} + \|B^\delta_0\|_{H^s} + T \sup_{t\in[0,T]}\|K_\delta J^\delta_l\|_{H^s}),
\end{equation}
\begin{equation}
\label{equ:time_esti_elecmag}
\sup_{t\in[0,T]}(\|\partial_t E^\delta_{l+1}\|_{H^{s-1}} + \|\partial_t B^\delta_{l+1}\|_{H^{s-1}}) \leq C (\sup_{t\in[0,T]}\|K_\delta J^\delta_{l}\|_{H^{s-1}} + \sup_{t\in[0,T]}(\|E^\delta_{l+1}\|_{H^s} + \|B^\delta_{l+1}\|_{H^s})),
\end{equation}
where $C>0$ is independent of $l$. Thus, by Sobolev embeddings, we have $(E^\delta_l, B^\delta_l)$ belong to $L^\infty((0,T)\times\R^3)$.

Once the electromagnetic field $(E^\delta_l,B^\delta_l)$ is known and belongs to $L^\infty(0,T;L^\infty(\Omega))$, it can be shown by arguments similar to those in \cite{Deg1986}, that there exists a unique non-negative solution $f^\delta_{l+1}\in L^2([0,T]\times\R^3_x,H^1(\R^3_v))$ to \eqref{equ:lin_regu_Vlasov}, satisfying the following properties:
\begin{align*}
f^\delta_{l+1}\in L^\infty(0,T;L^1\cap L^\infty(\R^3\times\R^3)),\\
\sup_{t\in[0,T]}\|f^\delta_{l+1}\|_{L^1(\R^3\times\R^3)} = \|f^\delta_0\|_{L^1(\R^3\times\R^3)},\quad
\sup_{t\in[0,T]}\|f^\delta_{l+1}(t)\|_{L^\infty(\R^3\times\R^3)} \leq e^{3T} \| f^\delta_0 \|_{L^\infty(\R^3\times\R^3)},\\
\sup_{t\in[0,T]}\int_{\R^3}\int_{\R^3}|v|^2 f^\delta_{l+1} \mathrm{d}x\mathrm{d}v \leq e^T \left(\int_{\R^3}\int_{\R^3} \dfrac{|v|^2}{2} f^\delta_{0}\mathrm{d}x\mathrm{d}v + CT\|E^\delta_l\|^2_{L^\infty}\right). 
\end{align*}

\textbf{Step $2$: Short-time solution to the regularized system.}
First, we prove that there exists a terminal time $T^\star$, sufficiently small, such that the iterative sequences remain uniformly
bounded in their respective function spaces with respect to $l$:
\begin{align*}
|v|^2 f^\delta_{l}\in L^\infty(0,T^\star; L^1(\R^3\times \R^3)),\\
E^\delta_l,\,\, B^\delta_l \in C([0,T^\star],H^s(\R^3))\cap C^1([0,T^\star],H^{s-1}(\R^3)).
\end{align*}
Indeed, we assume that the initial conditions satisfy 
\[\||v|^2 f^\delta_0\|_{L^1(\R^3\times\R^3)} \leq P,\quad \quad \|E^\delta_0\|_{H^s(\R^3)},\,\, \|B^\delta_0\|_{H^s(\R^3)}\leq Q/2, \]
for some real positive numbers $P,Q>0$. For the induction step, let us assume that there exists $T^\star>0$ such that  for all $t\in[0,T^\star]$
\[ \||v|^2 f^\delta_l\|_{L^1(\R^3\times\R^3)} \leq 2P,\quad \quad \|E^\delta_l\|_{H^s(\R^3)},\,\, \|B^\delta_l\|_{H^s(\R^3)}\leq 2Q.\] 
From the kinetic energy estimate obtained in Step $1$, we have, for any $0<T\leq T^\star$,	 that
$ \sup_{t\in[0,T]}\| |v|^2 f^\delta_{l+1}\|_{L^1(\R^3\times\R^3)} \leq e^T (P + CTQ^2)$.
Hence, there exists $0<T_1 < T$ such that 
\[ \sup_{t\in[0,T_1]} \||v|^2 f^\delta_{l+1}\|_{L^1(\R^3\times\R^3)}\leq 2P.\]
On the other hand, from the definition of $J^\delta_{l+1}$, we have
$ \sup_{t\in[0,T]} \| J^\delta_{l+1} \|_{L^1(\R^3)}\leq \|f^\delta_0\|_{L^1(\R^3\times\R^3)} + 2P$,
and combining with \eqref{equ:esti_elecmag}, we deduce that 
\begin{align*} 
\sup_{t\in[0,T]} (\|E^\delta_{l+1}\|_{H^s} + \|B^\delta_{l+1}\|_{H^s}) \leq e^{CT} (Q + C(\delta) T(\|f^\delta_0\|_{L^1(\R^3\times\R^3)} + 2P)).
\end{align*}
Hence, there exists $0< T_2 < T$ such that
\[ \sup_{t\in[0,T_2]}(\|E^\delta_{l+1}\|_{H^s} + \|B^\delta_{l+1}\|_{H^s}) \leq 2Q.\]
As a consequence, we get from \eqref{equ:time_esti_elecmag} that there exists a constant $C>0$ independently of $l$ such that 
\[ \sup_{t\in[0,T_2]}(\|\partial_t E^\delta_{l+1}\|_{H^s} + \|\partial_t B^\delta_{l+1}\|_{H^s}) \leq C.\]
Finally, by choosing $T^\star = \min(T_1,T_2)$, we finish the proof.

We are now ready to prove that the regularized system \eqref{equ:regu_Vlasov}-\eqref{equ:regu_FaAm} admits a local solution. To obtain this, it remains to show the compactness of the sequence of iterated solutions in the appropriate topology. 

Let $(f^\delta_l, E^\delta_l,B^\delta_l)$
be the sequence of solutions obtained from the iterated scheme \eqref{equ:lin_regu_Vlasov}-\eqref{equ:lin_initial} on $[0,T^\star]$. We will apply a result of compactness due to Aubin (see \cite{Simon}). Since $E^\delta_l,B^\delta_l\in C([0,T^\star];H^s(\R^3))$ and  $\partial_t E^\delta_l, \partial_t B^\delta_l\in C([0,T^\star];H^{s-1}(\R^3))$, uniformly in $l$, Aubin's lemma assure us that $E^\delta_l, B^\delta_l$ are relatively compact in $L^q(0,T^\star;H^{s-1}_{\text{loc}}(\R^3))$ for any $1\leq q<\infty$. On the other hand, since $f^\delta_l$ is uniformly bounded in $L^\infty(0,T^\star;L^\infty(\R^3\times\R^3))$ and using the velocity moment estimate from kinetic theory \cite{GolSaintMag1999}, we  obtain
\[\rho^\delta_l \in  L^\infty(0,T^\star;L^{5/3}(\R^3)),\quad J^\delta_l\in L^\infty(0,T^\star;L^{5/4}(\R^3))\]
uniformly in $l$. Hence, we have, up to the extraction of a subsequence:
\begin{align*}
f^\delta_l \rightharpoonup f^\delta \quad\text{weakly}-\star\quad\text{in}\quad L^\infty(0,T^\star;L^\infty\cap\mathcal{M}(\R^3\times\R^3)),\\
\rho^\delta_l \rightharpoonup \rho^\delta \quad\text{weakly}-\star\quad\text{in}\quad L^\infty(0,T^\star;L^{5/3}(\R^3)),\\
J^\delta_l \rightharpoonup J^\delta \quad\text{weakly}-\star\quad\text{in}\quad L^\infty(0,T^\star;L^{5/4}(\R^3)),\\
|v|^2 f^\delta_l \rightharpoonup F^\delta \quad\text{weakly}-\star\quad\text{in}\quad L^\infty(0,T^\star;\mathcal{M}(\R^3\times\R^3)),\\
E^\delta_l \rightharpoonup E^\delta \quad\text{strongly}\quad\text{in}\quad L^q(0,T^\star;L^2_{\text{loc}}(\R^3)),\\
B^\delta_l \rightharpoonup B^\delta \quad\text{strongly}\quad\text{in}\quad L^q(0,T^\star;L^2_{\text{loc}}(\R^3)).
\end{align*}
Using the boundedness of the kinetic energy, it is easy to show that  $\rho^\delta$, $J^\delta$ and $F^\delta$ are in fact $\int_{\R^3}f^\delta\mathrm{d}v$, $\int_{\R^3}v f^\delta\mathrm{d}v$ and $\int_{\R^3}|v|^2f^\delta\mathrm{d}v$, respectively.

Using the above properties, it is straightforward to pass to the limit in equation \eqref{equ:lin_regu_Vlasov} for $f^\delta_l$ and to show that $f^\delta$ verifies equation \eqref{equ:regu_Vlasov} with initial condition $f^\delta_0$. Also, $(E^\delta,B^\delta)$ is a solution of Maxwell's equations \eqref{equ:regu_FaAm}.

\textbf{Step 3: Global solution to the regularized system.}
Once we have a short time solution to the regularized system, we can use the energy argument to extend the solution to a global one.
\begin{lemma}
\label{regu_loc_ener}
Let $\delta>0$ be fixed. Then for the regularized system of equations \eqref{equ:regu_Vlasov}-\eqref{equ:regu_FaAm}, we have the following balance of energy identity:
\begin{equation}
\label{equ:regu_loc_ener}
\begin{split}
&\int_{\R^3}\int_{\R^3}\dfrac{|v|^2}{2}f^\delta\mathrm{d}x\mathrm{d}v +  \int_{\R^3}\left(\dfrac{|E^\delta|^2}{2} + \dfrac{|B^\delta|^2}{2}\right)\mathrm{d}x + \int_0^t \int_{\R^3}\int_{\R^3}|v|^2 f^\delta\mathrm{d}x\mathrm{d}v\mathrm{d}s\\
& = \int_{\R^3}\int_{\R^3}\dfrac{|v|^2}{2}f^\delta_0\mathrm{d}x\mathrm{d}v +  \int_{\R^3}\left(\dfrac{|E^\delta_0|^2}{2} + \dfrac{|B^\delta_0|^2}{2}\right)\mathrm{d}x + 3 \int_0^t\int_{\R^3}\int_{\R^3}f^\delta \mathrm{d}x\mathrm{d}v\mathrm{d}s,\,\,t\in[0,T^\star].
\end{split}
\end{equation} 
\end{lemma}
\begin{proof}
We note that $f^\delta\in L^\infty(0,T^\star;L^1\cap L^\infty(\R^3\times\R^3))$ and $E^\delta \in L^\infty((0,T^\star)\times\R^3)$. Using the same arguments as in \cite{BoniCarSoler1997} Lemma $3.7$, we have the balance of kinetic energy:
\begin{equation}
\label{equ:regu_loc_kin}
\begin{split}
 &\int_{\R^3}\int_{\R^3}\dfrac{|v|^2}{2}f^\delta\mathrm{d}v\mathrm{d}x + \int_0^t \int_{\R^3}\int_{\R^3}|v|^2 f^\delta\mathrm{d}v\mathrm{d}x\mathrm{d}s \\
&= \int_{\R^3}\int_{\R^3}\dfrac{|v|^2}{2}f^\delta_0 \mathrm{d}v\mathrm{d}x + \int_0^t \left< J^\delta,K_\delta E^\delta \right>\mathrm{d}s +  3 \int_0^t\int_{\R^3}\int_{\R^3}f^\delta \mathrm{d}v\mathrm{d}x\mathrm{d}s.
\end{split}
\end{equation}
Multiplying the Maxwell equation \eqref{equ:regu_FaAm} by $E^\delta$ and $B^\delta$, respectively, and integrating, we get the balance of electromagnetic energy:
\begin{equation}
\label{equ:regu_loc_elecmag}
 \int_{\R^3} \left( \dfrac{|E^\delta|^2}{2} + \dfrac{|B^\delta|^2}{2} \right)\mathrm{d}x =  \int_{\R^3} \left( \dfrac{|E^\delta|^2}{2} + \dfrac{|B^\delta|^2}{2} \right)\mathrm{d}x - \int_0^t \left<K^\delta J^\delta, E^\delta\right>\mathrm{d}s.
 \end{equation}
Since $K_\delta$ is self-adjoint, adding \eqref{equ:regu_loc_kin} and \eqref{equ:regu_loc_elecmag} together gives \eqref{equ:regu_loc_ener}.
\end{proof}

We now show that our local solution does not explode at any arbitrary $T>0$ in order to extend the solution from $[0,T^\star]$ to $[0,\infty)$. In particular, \eqref{equ:regu_loc_ener} yields an estimate for 
\[ \sup_{\delta >0}\int_{\R^3}\int_{\R^3} |v|^2 f^\delta(t) \mathrm{d}v\mathrm{d}x \leq C,\quad\text{for all}\quad t\in[0,T). \]
Hence, we get $\sup_{\delta>0}\|J^\delta\|_{L^1(\R^3)}\leq C$ for all $t\in [0,T)$ and $K_\delta J^\delta$ remains bounded in $C([0,T);H^s(\R^3))$. We can now use Theorem I in \cite{Kato} to get that $E^\delta, B^\delta$ remain bounded in $C([0,T);H^s(\R^3))$. Applying the argument of Degond in \cite{Deg1986}, we obtain that the unique solution $f^\delta$ is bounded in $L^\infty(0,T;L^1\cap L^\infty(\R^3\times\R^3))$. Our solution can therefore be extended by continuity to a global solution to the regularized system.

\textbf{Step 4: Properties of the solution.} The rest of Proposition \ref{regu_solu} is devoted to study the properties of the $(f^\delta,E^\delta, B^\delta)$ sequence of solutions. We note that the solution satisfies:
\begin{align*}
f^\delta\in L^\infty(0,T;L^1\cap L^\infty),\quad |v|^2 f^\delta \in L^\infty(0,T;L^1),\\
E^\delta,\,\, B^\delta \in L^\infty(0,T;L^2\cap L^\infty).
\end{align*}
Then, using the same arguments as in \cite{BoniCarSoler1997}, we obtain:
\begin{enumerate}
\item Balance of mass and $L^2$ identities:
\[  \int_{\R^3}\int_{\R^3} f^\delta \mathrm{d}v\mathrm{d}x = \int_{\R^3}\int_{\R^3} f^\delta_0\mathrm{d}v\mathrm{d}x.\]
\[
 \|f^\delta\|^2_{L^2(\R^3\times\R^3)} + \|\nabla_v f^\delta\|^2_{L^2((0,T)\times\R^3\times\R^3)}=  \|f^\delta_0\|^2_{L^2(\R^3\times\R^3)} + 3     \int_0^t \|f^\delta\|^2_{L^2(\R^3\times\R^3)}\mathrm{d}s.
\]
\item Balance of energy \eqref{equ:regu_loc_ener} which can be rewritten as
\begin{align*}
&\int_{\R^3}\int_{\R^3}\dfrac{|v|^2}{2}f^\delta\mathrm{d}x\mathrm{d}v +  \int_{\R^3}\left(\dfrac{|E^\delta|^2}{2} + \dfrac{|B^\delta|^2}{2}\right)\mathrm{d}x \\
&= \int_{\R^3}\int_{\R^3}\dfrac{|v|^2}{2}f^\delta_0\mathrm{d}x\mathrm{d}v +  \int_{\R^3}\left(\dfrac{|E^\delta_0|^2}{2} + \dfrac{|B^\delta_0|^2}{2}\right)\mathrm{d}x + \int_0^t\int_{\R^3}\int_{\R^3} (\nabla_v f^\delta + v f^\delta)\cdot v \mathrm{d}v\mathrm{d}x\mathrm{d}s,
\end{align*}
and entropy identities:
\[
\int_{\R^3}\int_{\R^3} f^\delta \ln f^\delta \mathrm{d}v\mathrm{d}x  = \int_{\R^3}\int_{\R^3} f^\delta_0 \ln f^\delta_0 \mathrm{d}v\mathrm{d}x + \int_0^t\int_{\R^3}\int_{\R^3} (\nabla_v f^\delta + v f^\delta)\cdot\dfrac{\nabla_v f^\delta}{f^\delta} \mathrm{d}v\mathrm{d}x\mathrm{d}s.
\]
\item The local conservation laws in the sense of distributions.
\end{enumerate}

\begin{proof}(of Theorem \ref{exis_weak_sol})

\textbf{Compactness.} In this part, we complete the proof of the existence of a solution to Theorem \ref{exis_weak_sol} by passing to the limit $\delta\to 0$. We begin by recalling the following result due to Bouchut–Dolbeault \cite{BouDol1995}.
\begin{lemma}
\label{Compact_lemma}
Denote by F and H two bounded subsets of  $L^p(\Omega\times\R^3)$ and $L^q(0,T;L^p(\Omega\times\R^3))$, respectively, with $1\leq p<\infty$ and $1<q\leq \infty$. Consider the solution $g\in C([0,T];L^p(\Omega\times\R^3))$ of
\[
\partial_t g + v\cdot \nabla_x g - \Divv(vg) - \Delta_v g = h,\,\,(t,x,v)\in(0,T)\times\Omega\times\R^3, \quad\quad g(0,x,v)= g_0(x,v),
\]
with initial data $g_0\in F$ and $h\in H$. Then for any $\tau>0$ and $\omega$ bounded open subset of $\Omega\times\R^3$, $g$ is compact in $C([\tau,T],L^p(\omega))$.
\end{lemma}

We write the Vlasov equation \eqref{equ:regu_Vlasov} in the form
\begin{equation}
\label{equ:rew_Vlasov}
\partial_t f^\delta + v\cdot \nabla_x f^\delta -\Divv(vf^\delta - (v\wedge \textbf{B}_{\text{ext}})f^\delta) - \Delta_v f^\delta = -K_\delta(E^\delta + v\wedge B^\delta)\cdot\nabla_v f^\delta.
\end{equation}
Notice that from Step $4$, we have $f^\delta\in L^\infty(0,T;L^2(\R^3\times\R^3))$,  $\nabla_v f^\delta\in L^2((0,T)\times\R^3\times\R^3)$ are bounded, and $E^\delta,\,\,B^\delta\in L^\infty(0,T;L^2(\R^3))$ are also bounded, uniformly in $\delta>0$. Hence the right-hand side of \eqref{equ:rew_Vlasov} is clearly bounded in $L^1((0,T)\times\R^3\times B_R)$ for all $T>0$ and $R>0$. Lemma \ref{Compact_lemma} then implies that the sequence distribution $(f^\delta)_{\delta>0}$ is compact in $L^1_{\text{loc}}$. Therefore, up to extraction of a subsequence (still denote $\delta$), we have
\begin{align*}
& f^\delta\to f \quad\text{strongly in}\quad L^1_{\text{loc}}(\R_+\times\R^3\times\R^3),\\
& f^\delta\rightharpoonup f\quad\text{weakly in}\quad L^\infty(0,T;L^2(\R^3\times\R^3)),\\
& E^\delta \rightharpoonup E \quad\text{weakly in}\quad L^\infty(0,T;L^2(\R^3)),\\
& B^\delta \rightharpoonup B \quad\text{weakly in}\quad L^\infty(0,T;L^2(\R^3)).
\end{align*}
As a consequence, we have for any $\psi\in \mathcal{D}(\R^3)$ and $R>0$ that
\[\int_{\R^3} f^\delta \psi\mathrm{d}v,\quad \int_{\R^3}f^\delta v\psi\mathrm{d}v \to  \int_{\R^3} f \psi \mathrm{d}v,\quad \int_{\R^3} f v \psi\mathrm{d}v\quad\text{in}\quad L^1((0,T)\times  B_R).\]
On the other hand, using the uniform bound of $(f^\delta)_{\delta>0}$ in $L^\infty(0,T;L^2(\R^3\times\R^3))$ and the Dunford-Pettis theorem yields easily (see \cite{DipLion}):
\[
\left| \int_{\R^3} f^\delta \psi \mathrm{d}v \right|^2,\quad \left| \int_{\R^3}f^\delta v \psi\mathrm{d}v \right|^2 \quad\text{are weakly compact in} \quad L^1((0,T)\times B_R).
\]
Hence, using the Vitali convergence theorem, we deduce that
\begin{equation}
\label{equ:strong_L2}
\int_{\R^3} f^\delta \psi\mathrm{d}v,\quad \int_{\R^3}f^\delta v\psi\mathrm{d}v \to \int_{\R^3} f \psi\mathrm{d}v,\quad \int_{\R^3}f v\psi\mathrm{d}v\quad\text{in}\quad L^2((0,T)\times B_R).
\end{equation}

\textbf{Passing to the limit.}
It remains to show that $(f,E,B)$ is a weak solution of \eqref{equ:VMFP_noeps}-\eqref{equ:init_data_noeps}. Thanks to the convergences obtained above, it suffices to show the convergence towards $0$ for any function test $\Psi(t,x,v) = \varphi(t,x)\psi(v)\in \mathcal{D}([0,T)\times\R^3\times\R^3)$ of the non-linear contributions:
\begin{align*}
 \int_0^T \int_{\R^3}\varphi(t,x) K_\delta E^\delta \left(\int_{\R^3} f^\delta \psi\mathrm{d}v\right)\mathrm{d}x\mathrm{d}t\to \int_0^T\int_{\R^3}\varphi(t,x)E\left(\int_{\R^3}  f\psi\mathrm{d}v\right)\mathrm{d}x\mathrm{d}t =0,\\
\int_0^T \int_{\R^3}\varphi(t,x) K_\delta B^\delta \left(\int_{\R^3} f^\delta v\psi\mathrm{d}v\right)\mathrm{d}x\mathrm{d}t\to \int_0^T\int_{\R^3}\varphi(t,x)B\left(\int_{\R^3}  f v\psi\mathrm{d}v\right)\mathrm{d}x\mathrm{d}t =0.
\end{align*}
This follows from \eqref{equ:strong_L2} and $\varphi K_\delta E^\delta ,\,\,\varphi K_\delta B^\delta$ converge weakly to $E\varphi,\,\ B\varphi$, respectively, in $L^2((0,T)\times B_R)$.
\end{proof} 



\end{document}